\documentclass[10pt]{amsart}
\usepackage{amsmath, amssymb, amsthm, amsfonts}
\usepackage{graphicx} % Required for inserting images
\usepackage{tikz, tikz-cd}
\usepackage{enumerate}
\usepackage[normalem]{ulem}

\makeatletter
\@namedef{subjclassname@2020}{\textup{2020} Mathematics Subject Classification}
\makeatother

\usepackage{hyperref}
\hypersetup{
    colorlinks=true,
    citecolor=blue,
    linkcolor=blue,
    filecolor=magenta,      
    urlcolor=blue,
}

\numberwithin{equation}{section}

\newtheorem{thm}{Theorem}[section]

\newtheorem{cor}[thm]{Corollary}
\newtheorem{lem}[thm]{Lemma}
\newtheorem{prop}[thm]{Proposition}

\theoremstyle{definition}

\newtheorem{defn}[thm]{Definition}
\newtheorem{ex}[thm]{Example}

\newtheorem{rmk}[thm]{Remark}
\newtheorem{setup}[thm]{Set-up}

\newcommand{\Aut}{\mathrm{Aut}}
\newcommand{\CH}{\mathrm{CH}}
\newcommand{\ch}{\mathrm{ch}}
\newcommand{\Cl}{\mathrm{Cl}}

\newcommand{\Def}{\mathrm{Def}}
\newcommand{\Ext}{\mathrm{Ext}}

\newcommand{\Hom}{\mathrm{Hom}}

\newcommand{\Ind}{\mathrm{Ind}}
\newcommand{\Irr}{\mathrm{Irr}}

\newcommand{\rk}{\mathrm{rk\,}}
\newcommand{\Res}{\mathrm{Res}}
\newcommand{\Sing}{\mathrm{sing}}
\newcommand{\sm}{\mathrm{sm}}
\newcommand{\Supp}{\mathrm{Supp\,}}

\newcommand{\CC}{\mathbb{C}}

\newcommand{\PP}{\mathbb{P}}
\newcommand{\QQ}{\mathbb{Q}}

\newcommand{\ZZ}{\mathbb{Z}}

\newcommand{\sA}{\mathcal{A}}

\newcommand{\sC}{\mathcal{C}}

\newcommand{\sE}{\mathcal{E}}
\newcommand{\sF}{\mathcal{F}}

\newcommand{\sI}{\mathcal{I}}

\newcommand{\sK}{\mathcal{K}}
\newcommand{\sL}{\mathcal{L}}
\newcommand{\sM}{\mathcal{M}}
\newcommand{\sO}{\mathcal{O}}

\newcommand{\sS}{\mathcal{S}}
\newcommand{\sT}{\mathcal{T}}

\newcommand{\tC}{\widetilde{C}}
\newcommand{\tD}{\widetilde{D}}

\newcommand{\tG}{\widetilde{G}}

\newcommand{\tP}{\widetilde{P}}

\newcommand{\tS}{\widetilde{S}}
\newcommand{\tT}{\widetilde{T}}

\newcommand{\tX}{\widetilde{X}}

\newcommand{\tZ}{\widetilde{Z}}

\newcommand{\oG}{\overline{G}}
\newcommand{\oH}{\overline{H}}

\newcommand{\oK}{\overline{K}}

\newcommand{\oM}{\overline{M}}

\newcommand{\oS}{\overline{S}}
\newcommand{\oT}{\overline{T}}

\newcommand{\tor}{\mathrm{tor}}

\newcommand{\chara}{\mathrm{char}}

\newcommand{\reg}{\mathrm{reg}}
\newcommand{\td}{\mathrm{td}}
\newcommand{\tr}{\mathrm{tr}}

\begin{document}

\title[cohomological representations of finite automorphism groups]{On the cohomological representations of finite automorphism groups of singular curves and compact complex spaces}

\author{Qing Liu}
\author{Wenfei Liu}

\address{Univ. Bordeaux, CNRS, IMB, UMR 5251, F-33400 Talence, France}
\email{qing.liu@math.u-bordeaux.fr}

\address{School of Mathematical Sciences, Xiamen University, Siming South Road 422, Xiamen, Fujian Province, P.~R.~China}
\email{wliu@xmu.edu.cn}

\date{January 9, 2026}

\begin{abstract}
Let $G$ be a finite group acting tamely on a proper reduced curve $C$
over an algebraically closed field $k$.    We study the $G$-module
structure on the cohomology groups of a $G$-equivariant locally free 
 sheaf $\sF$ on $C$ and give formulas of Chevalley--Weil type, with values in the Grothendieck ring $R_k(G)_\QQ$ of finitely generated $G$-modules.  
We also provide a similar formula for the singular cohomology of compact complex spaces.

Our focus is on the case where $C$ is nodal. Using the Chevalley--Weil formula, we compute the $G$-invariant part $H^0(C, \omega_C^{\otimes m})^G$ for $m\geq 1$. In turn, we use the formula for $\dim H^0(C, \omega_C^{\otimes 2})^G$ to compute the equivariant deformation space of a stable $G$-curve $C$. We also obtain numerical criteria for the presence of any given irreducible representation in $H^0(C, \omega_C\otimes \sA)$, where $\sA$ is an ample locally free $G$-sheaf on $C$. Some new phenomena, pathological compared to the smooth curve case, are discussed. 
\end{abstract}

\thanks{The second author was supported by the NSFC (No.~12571046).}

\maketitle

\tableofcontents

\section{Introduction} 
Let $X$ be a projective variety over a field $k$, let $G$ be a finite
group acting on $X$, and $\sF$ a $G$-equivariant sheaf (also called a $G$-sheaf) of
$k$-vector spaces. Then $G$ acts on the cohomology groups $H^*(X, \sF)$, and it is a natural and important problem to determine $H^*(X, \sF)$ as a $G$-representation. The formulas of Chevalley--Weil \cite{CW34, Wei35} and Broughton \cite{Bro87} achieve this for the case where $X$ is a complex smooth projective curve and $\sF$ is a pluricanonical sheaf $\omega_X^{\otimes m}$ with $n\geq 1$ or the locally constant sheaf $\CC_X$ respectively.\footnote{In retrospect, Broughton's formula for $H^1(X, \CC)$, where $X$ is a smooth projective curve over $\CC$, follows easily from Chevalley--Weil's formula of $H^0(X, \omega_X)$ together with the Hodge decomposition of $H^1(X, \CC)$.} 

The key feature of the Chevalley--Weil formulas in \cite{CW34, Wei35} is that, as $G$-representations, the cohomology groups $H^0(C, \omega_C^{\otimes m})$ for the pluricanonial sheaves of a complex smooth projective $G$-curve $C$ can be expressed as a rational multiple of the regular representation $\CC[G]$ plus a correction term determined by the inertia locus of the $G$-action. It is in the same spirit of the holomorphic Lefschetz fixed point theorem, as developed by \cite{AS68} for compact complex manifolds with a finite group action and  generalized by \cite{Don69} to the algebraic setting. 

This observation was made by Ellingsrud--L\o nsted \cite{EL80}, who used the idea of equivariant K-theory to set up a fairly general framework for studying the $G$-module structure on the cohomology groups of $G$-equivariant sheaves on a projective $G$-variety $X$ over any algebraically closed base field $k$. They define the $G$-equivariant Euler characteristic (called Lefschetz trace in \cite{EL80})
\[
\chi_G(\sE)=\sum_i (-1)^i[H^i(X,\sE)]
\]
for locally free $G$-sheaves $\sE$, as an element in the Grothendieck
ring $R_k(G)$ of $G$-modules.   However, when writing down an explicit formula, one needs to restrict to the case of
a smooth projective curve. Later, Kani and Nakajima \cite{Kan86,
  Nak86b} introduced the notion of ramification module to encode the
contribution of the fixed locus. K\"ock \cite{Koc05} clarified the connection of the Chevalley--Weil formula for locally
free $G$-sheaves on a smooth projective curve to the holomorphic
Lefschetz fixed point theorems of \cite{AS68} and  \cite{Don69}. There
is on-going interest in refining and reproving the Chevalley--Weil
formula for smooth curves, or spelling it out more explicitly
in specific situations (see, e.g., \cite{Tam51, MV81, Hur92,
  Bor03, Bor06, Hor12, KK12, KK13, Can18, KFW19, BCK20, CKKK22,
  Ara22, BW23, Gar23,   CL24, LK24}).

Yubo Tong \cite{Ton25} made a first attempt to write down a Chevalley--Weil type formula
for the canonical sheaf $\omega_C$ of a \emph{nodal} curve $C$ with a faithful tame
action of a finite group $G$. However, his assumption that the
quotient curve is smooth is quite restrictive. 
In fact, if $G$ is cyclic, then the smoothness condition on the
quotient curve forces the order $|G|$ to be at most $2$. Also,
he dealt only with the canonical sheaf $\omega_C$.

In the present work, we treat the general case of locally free $G$-sheaves $\sE$ on a proper reduced singular $G$-curve $C$.

\begin{thm}[{= Theorem~\ref{thm: CW singular}}]\label{thm: CW}
Let $C$ be a proper reduced curve over an algebraically closed field
$k$, and $G$ a finite group acting tamely on
  $C$ (see Definition~\ref{defn:tame_action}).  Then for 
any locally free $G$-sheaf $\sE$ of rank $r$ on $C$, we
have the following equality in $R_k(G)_\QQ$: 
\begin{equation}\label{eq: CW intro}
\chi_G(\sE)= \left(\frac{1}{|G|}\chi(\sE)\right)[k[G]] + \sum_{Z\subset C}\Gamma_G(\sE)_Z,
\end{equation}
where $Z$ ranges over all irreducible subvarieties of $C$ with nontrivial inertia group $G_Z$ as well as the singular points of $C$, and the ramification modules $\Gamma_G(\sE)_Z$ are defined in
Definitions~\ref{defn: ramification module sm} and \ref{defn: ramification module singular}. 
\end{thm}

The key part of the formula \eqref{eq: CW intro} is the definition of
the ramification modules $\Gamma_G(\sE)_Z$, which are elements of $R_k(G)_\QQ$ of virtual degree $0$. When the curve $C$ is connected, smooth, and the $G$-action is faithful,  $\Gamma_G(\sE)_P$ is obtained by twisting the expression of Atiyah--Singer--K\"ock type
\[
-\frac{1}{|G|}\Ind_{G_P}^G\sum_{d=0}^{|G_P|-1}d\theta_P^d\otimes [\sE|_P],
\]
as discovered in \cite{Koc05}, with "the same positive amount of" the regular representation; see Definition~\ref{defn: ramification module sm}. In general, there are additional contributions from the singularities of $C$ and from the components of $C$ with nontrivial inertia group; see Definition~\ref{defn: ramification module singular}.

Comparing $\chi_G(\sE)$ with $r\chi_G(\sO_C)$, we propose a $G$-equivariant degree $\deg_G\sE$ of $\sE$ (Definition~\ref{defn: deg_G E}) and give an   equivariant  Riemann--Roch type formula that generalizes
  \cite[Th\'eor\`eme 4.10]{Bor03}: 
  \begin{thm}[ = Corollary~\ref{cor: Equiv-RR}] Under the assumption of
    Theorem~\ref{thm: CW}, we have the following equality in $R_k(G)_\QQ$:
\[
\chi_G(\sE)=r\chi_G(\sO_C)+\deg_G \sE.
\]
  \end{thm}

Let us comment on the difficulties one encountered when dealing with
automorphisms of singular curves. One 
complication coming up in the reducible curve
case is that a nontrivial subgroup of $\Aut(C)$ can act trivially on an entire component
of $C$, so  the action of a group $G$ can be faithful but not free on a
dense open subset of $C$. In general, the order $|\Aut(C)|$ for a
stable curve $C$ of genus $g$ can only be bounded by an exponential
function of $g$ (\cite{vOpV07}), in strong contrast to the
Hurwitz bound $84(g-1)$ for connected smooth genus $g$ curves
(\cite{Hur92}). Recall that a stable curve is a
connected nodal curve $C$, with finite automorphism group $\Aut(C)$, or equivalently, with ample canonical sheaf $\omega_C$.

We were unable to adapt directly the method of \cite{EL80} to the singular case, because the crucial
exact sequence (3.1) therein does not hold in the singular case
(and in fact neither in higher-dimensional smooth case). Instead, we use the natural idea of comparing the cohomological groups with their pull-back to
the normalization $\nu\colon\tC\to C$ of $C$. To fix ideas, take a locally free $G$-equivariant sheaf $\sE$ on $C$. Then we have
\[
0\rightarrow  \sE\rightarrow \nu_*\nu^*\sE \rightarrow \sE|_S \rightarrow 0
\]
where $S$ is the singular locus of $C$ endowed with the reduced
  structure. Here we encounter two $G$-equivariant sheaves $\nu^*\sE$
and $\sE|_S$ on usually disconnected varieties $\tC$ and $S$
respectively. The simple and useful observation is that the cohomology
representation of a $G$-equivariant sheaf $\sE$ on a disconnected
variety $X$ can be induced from the cohomology representation of
$G_{\{X_j\}}$ on $H^*\left(X_j, \sE|_{X_j}\right)$, where the $X_j$'s
are the connected components of $X$, and $G_{\{X_j\}}\subseteq G$ is the
stabilizer of $X_j$ (see Lemma~\ref{lem: conn decomp}).

After writing down the Chevalley--Weil formula, we address the following natural questions for the cohomology of a $G$-sheaf $\sF$ on a $G$-curve $C$: 
\begin{enumerate}
    \item Is there a positive rational number $a$ such that $\chi_G(\sF)\geq a[k[G]]$ holds in $R_k(G)_\QQ$? If $G$ is reductive over $k$, then this is equivalent to asking  wheter every irreducible $G$-module has positive virtual multiplicity in $\chi_G(\sF)$.
    \item Is $H^*(C, \sF)$ a faithful representation of $G$?
        \item What is the $G$-invariant part $H^*(C, \sF)^G$?
\end{enumerate}
Of course, the answers depend on $\sF$ and $C$. We focus on the case of nodal $G$-curves, and the $G$-sheaves are taken to be related to the canonical sheaf $\omega_C$. 

We first compute the $G$-invariant global sections of the pluricanonical sheaves on a stable pointed $G$-curve.
\begin{thm}[{=  Corollary~\ref{cor: pluricanonical G-inv}}]
Let $C$ be a connected proper nodal curve over an algebraically closed
field $k$, and $G\subset\Aut(C)$ a  finite subgroup of automorphisms
such that $|G|$ is not divisible by $\chara(k)$. Let $S$ be the singular locus of $C$, and $T$ a (possibly empty) finite $G$-set contained in the smooth locus of $C$ such that $\omega_C(T)$ is ample. Then for any positive integer $m$, we have
\begin{multline*}
\dim_k H^0(C, \omega_C(T)^{\otimes m})^G =\dim_k H^0(D, \omega_D(\oT)^{\otimes m}) +(m - \epsilon_m) \#\oS_2 \\
+\sum_{Q\in D\setminus (\oS\cup\oT)}\left\lfloor m\left(1-\frac{1}{|e_{Q}|}\right)\right\rfloor  
\end{multline*}
where $\pi\colon C\rightarrow D=C/G$ is the quotient map, $\oS,\, \oT\subset D$ are images of $S,\, T$ respectively,  $e_Q:=|G_P|$ for $Q\in D\setminus(\oS\cup \oT)$ and $P\in \pi^{-1}(Q)$, and $\epsilon_m=1$ if $m$ is odd and $0$ if $m$ is even.
\end{thm}
In Section~\ref{sec: Def(C,G)}, we apply the obtained formula for
$\dim H^0(C, \omega_C^{\otimes 2})^G$ to compute the dimension the
$G$-equivariant deformation space $\Def(C, G)$ of a stable
$G$-curve. This is useful in the study of the singularities of the
moduli spaces $\oM_g$ of stable curves.  In forthcoming work, we
will analyze further the condition for the faithfulness of the action
of $G$ on $H^0(C, \omega_C^{\otimes m})$.

If $C$ is nodal and each component of $C$ has trivial inertia group,
we obtain numerical criteria for $H^0(C, \omega_C\otimes\sA)$,
where $\sA$ is an ample locally free sheaf, to contain a rational multiple
of the regular representation  (Lemma~\ref{lem: chi_G omega + A} and
Corollary~\ref{cor: H0 omega tC tensor A vs k[G]}). 

It is well known that the dual graph  is very useful in describing the structure of a nodal curve $C$ (\cite[Chapter X]{ACG11}). We observe that there is an induced action of $G$ on the dual graph $\Gamma:=\Gamma(C)$ of $C$, viewed as a 1-dimensional CW complex. We use $\chi_G(\Gamma):= [H_0(\Gamma, k)] -  [H_1(\Gamma, k)]\in R_k(G)$ to describe the discrepancy between $H^0(C, \omega_C)$ and $H^0(\tC, \omega_{\tC})$, where $\nu\colon \tC\rightarrow C$ is the normalization (Corollary~\ref{cor: H^0(omega_C) vs H^0(omega_tC)}):
\[
[H^0(C, \omega_C)] = [H^0(\tC, \omega_{\tC})] + [1_G] - \chi_G(\Gamma, k)
\]
This can be viewed as the $G$-equivariant version of the usual relation between the arithmetic genus of $C$ and the genus of $\tC$.

We note that Broughton's formula mentioned in the beginning of the introduction may be proved by a straightforward topological argument, and hence can be vastly generalized, as follows.
\begin{thm}[{= Theorem~\ref{thm: CW for CW}}]\label{thm: CW for CW intro}
Let $X$ be a possibly disconnected $n$-dimensional compact complex space, and $G\subset\Aut(X)$ a finite subgroup of automorphisms. Let $Y=X/G$ be the quotient space and $\pi\colon X\rightarrow Y$ the quotient map. For each subgroup $H<G$, denote
\[
X_H:=\{ x \in X \mid \text{$G_ x$ conjugate to $H$} \} \, \text{ and }\, Y_H:=\pi(X_H).
\]
Let $H_1,\dots, H_s$ be a system of representatives of conjugacy classes in $G$.
Then the following equality holds in $R_\CC(G)_\QQ$:
%% of rational $G$-representations:
%% \begin{equation}\label{eq: CW for CS_intro}
\[    \sum_{0\leq i\leq 2n}(-1)^i[H^i(X,\QQ)] = \sum_{1\leq j\leq s} \chi(Y_{H_j})\Ind_{H_j}^G [1_{H_j}].
\]
% \end{equation}
where for a complex space $Z$, $\chi(Z)$ denotes its topological Euler characteristic, and we set $\chi(Y_{H_j})=0$ if $Y_{H_j}=\emptyset$.
\end{thm}

Specializing Theorem~\ref{thm: CW for CW intro} to nodal curves, we
obtain

\begin{thm}[{= Corollary~\ref{cor: chi Q irreducible}}]\label{thm: Broughton intro}
  Let $C$
  be a proper nodal curve over $\CC$ and $G\subset\Aut(C)$ a finite subgroup
  of automorphisms acting freely on a dense open subset of $C$.
   Let $\pi\colon C\rightarrow D:=C/G$ be the quotient map and write $D=\cup_{1\leq j\leq m} D_j$ as the union of irreducible components. Let $S$ be the singular locus of $C$, and set
\[
A:=S\cup \{P\in C\mid G_P \text{ is nontrivial}\}
\]
Then
\[
\sum_{i=0}^{2}(-1)^i[H^i(C,\QQ)] = \sum_{j=1}^m \chi(D_j\setminus\pi(A))\left[\QQ[G]\right] + \sum_{P\in A'} \Ind_{G_{P}}^G \left[1_{G_{P}}\right].
\]
where $A'$ denotes a set of representatives of the $G$-orbits in $A$.
\end{thm}

Finally, we note that the Chevalley--Weil formula is important for computing of cohomological invariants of the so-called product-quotient  varieties $(C_1\times \dots \times C_n)/G$, where $C_1,\dots, C_n$ are smooth projective curves and $G$ is a finite group acting diagonally on their product (\cite{CLZ13, Gle16, FGP20}); this class of varieties is very useful for the construction and classification of projective varieties with prescribed invariants. Our Chevalley--Weil type formulas facilitates the computation of cohomological invariants of $(C_1\times \dots \times C_n)/G$ when the $C_i$ are nodal curves. These possibly non-normal varieties are in turn useful in the compactification of moduli spaces of varieties of general type; see \cite{vOp06, Liu12} for the surface case. 

\medskip

\noindent{\bf Notation and Conventions.}
Let us fix an algebraically closed field $k$. 

\begin{itemize}

\item A \emph{variety} in this paper means a reduced scheme of finite type over $k$, and we only consider the closed points. Thus a variety can be reducible or even disconnected. We will denote by $\Aut(X):=\Aut_k(X)$ the automorphism group of a variety $X$ over $k$. For a coherent sheaf $\sF$ on $X$ and a closed subscheme $Z\subset X$, we denote by $\sF|_Z:=\sF\otimes_{\sO_X}\sO_Z$ the restriction of $\sF$ to $Z$.

    \item Let $G$ be a finite group acting on a set $X$. For any subset $Z$ of $X$, we denote
\[
G_{\{Z\}}:=\{g\in G\mid g(Z)\subseteq Z\}
\]
the \emph{stabilizer or decomposition group} of $Z$ and
\[
G_Z:=\{g\in G\mid g(s) =s \text{ for any $s\in Z$}\}
\]
the \emph{inertia group of $Z$}. 
Obviously, we have $G_Z\subseteq G_{\{Z\}}$. For $P\in X$, its orbit $\{g(P)\mid g\in G\}$ is denoted by $G.P$. Conversely, the \emph{inertia locus} or \emph{fixed locus} of $(X, G)$ is the set of $x\in X$ with nontrivial stabilizer $G_x$.
\item Let $X$ be an algebraic variety and $G$ a finite group acting on $X$ by regular automorphisms. Then the inertia locus of $(X,G)$ is a closed subset of
$X$. If $G$ acts freely on a dense open subset of $X$ and if $X$ is 
quasi-projective, the inertia locus  
coincides with the ramification locus of the quotient map
$X\to X/G$. Note that if $G$ acts trivially on an irreducible component
of $X$, then $X\to X/G$ is \'etale on this component, while the
inertia group of this component is $G$, hence nontrivial in general. 
\end{itemize}

\section{Preliminaries}\label{sec: prelim}

For the reader's convenience, we recall or give some general facts which will be used later. In particular, we carefully recall facts about the Grothendieck ring of $k[G]$-modules, the Frobenius reciprocity and Brauer character in the representation theory of finite groups, in order to handle the case where the order of the group is divisible by the characteristic exponent of the base field $k$.

\subsection{Representation theory of finite groups}
\subsubsection{The Grothendieck ring of finitely generated \texorpdfstring{$k[G]$}{}-modules}\label{sec: Grothendieck ring}

Most of the material in this subsection is taken from
\cite{Ser77}. For a finite group $G$ and a field $k$, let $\Irr_k(G)$ be the set of irreducible
$G$-representations over $k$. The one-dimensional trivial
representation of $G$ is denoted by $1_G$ and called the
\emph{unit representation} of $G$.  

In this paper, we only consider \emph{finitely generated} $k[G]$-modules. The set of finite dimensional representations of $G$ are in a natural bijection with the set of finitely generated left $k[G]$-modules, and we use both terminologies interchangeably. Note that $k[G]$, viewed as a module over itself, corresponds to the \emph{regular representation}. 

We recall the fundamental Frobenius reciprocity: For a subgroup $H\subset G$, a $k[H]$-module $V$, and a $k[G]$-module $W$, let $\Ind_H^G V$ and $\Res^G_H W$ be the induction of $V$ to $G$ and the restriction of $W$ to $H$. Then there are canonical isomorphisms (\cite[Section~2.3]{Sch13}, first and second Frobenius reciprocities):
\begin{equation}\label{eq: Frobenius}
\begin{split}
 \Hom_G(\Ind_H^G V, W)\cong \Hom_H(V, \Res^G_H W),  \\
 \Hom_G(W, \Ind_H^G V)\cong \Hom_H(\Res^G_H W, V).
\end{split}
\end{equation}

We may form the Grothendieck group $R_k(G)$ of finitely generated
$k[G]$-modules (also called $G$-modules in the sequel if the base field $k$ is clear). For a
$G$-module $V$, we denote $[V]_{R_k(G)}$ (or simply $[V]$ if $G$ and $k$ are clear from the context)
its class in $R_k(G)$. Then $R_k(G)$ is a free abelian group with basis (\cite[page 115, Proposition~40]{Ser77})
\[
S:=\{[V]\mid V\in \Irr_k(G)\}.
\]
Two $G$-modules $V$ and $W$ has the same image in $R(G)$ if and only if they have the same composition factors (counted with multiplicity). 

We define the following operations on $R_k(G)$:
\begin{itemize}
\item $[V]\otimes [W]:=[V\otimes_k W]\in R_k(G)$, the tensor product
  of two $G$-modules. 
\item $\Res^G_H [V]:=\left[\Res^G_H V\right]\in R_k(H)$, the
restriction of $[V]$ with respect to a group homomorphism
      $\varphi\colon H\rightarrow G$. To simplify notation, we often 
      omit the symbol $\Res^G_H$ when the homomorphism $\varphi$ is clear.
\item $\Ind_G^{\tG}[V]:= \left[\Ind_G^{\tG} V\right]=\left[k[\tG]\otimes_{k[G]} V\right]\in R_k(\tG)$ the
  class of the induced representation if $G\subseteq \tG$ is a
  subgroup. 
\end{itemize}
These operations extend by linearity to the whole Grothendieck group $R_k(G)$.
In particular, $R_k(G)$ becomes a commutative ring with respect to the
tensor product, with unit $[1_G]$. For any commutative ring $A$ with unit, we define $R_k(G)_A := R_k(G)\otimes_\ZZ A$.

For $\alpha\in R_k(G)_\QQ$, we may write uniquely
\[
\alpha=\sum_{V\in \Irr_k(G)} \mu_V(\alpha)[V],\quad \mu_V(\alpha)\in\QQ.
\]
We call $\mu_V(\alpha)$ the \emph{virtual multiplicity} of $V$ in $\alpha$. The \emph{virtual degree} of $\alpha\in R_k(G)_\QQ$ is defined as
\[
\deg\alpha := \sum_{V\in \Irr_k(G)}\mu_V(\alpha)\cdot \dim V. 
\]
Define the subgroup
\[
R_k(G)^{+}_\QQ:=\{\alpha\in R_k(G)_\QQ \mid \mu_V(\alpha) \geq 0
\text{ for all } V\in \Irr_k(G)\}. 
\]
It induces a partial order on $R_k(G)_\QQ$ defined by 
$\alpha\geq \beta$ if $\alpha-\beta \in R_k(G)^{+}_\QQ$.

\subsubsection{Brauer character}\label{sec: character}
Let $G$ be a finite group, $k$ an algebraically closed field, and $\mu_\infty(k)$ the group of roots of unity in $k$. Set
\begin{equation}\label{eq: G_reg}
  G_{\reg}:=
\begin{cases}
    G & \text{if $\chara(k)=0$}\\
    \{g\in G\mid \text{the order of $g$ is prime to $p$}\} &\text{if $\chara(k)=p>0$}
\end{cases}  
\end{equation}
Then $G_{\reg}$ is a union of conjugacy classes. The elements of $G_{\reg}$ are called \emph{$p$-regular} if $\chara(k)=p>0$.

For a $k[G]$-module $V$ 
of degree $d$, let $\lambda_1,\dots,\lambda_d$ be the eigenvalues of any $g\in G_\reg$ on $V$, which are elements of $\mu_\infty(k)$. 
The \emph{Brauer trace} of $g$ on $V$ is
\begin{equation}\label{eq: Brauer trace}
\tr_B(g; V):=\sum_{i=1}^d \varphi(\lambda_i)\in \CC
\end{equation}
where $\varphi \colon \mu_\infty(k)\to \mu_\infty(\CC)$ is a given
injective homorphism.
\footnote{If $\chara(k)=0$, we may take directly an injective homorphism $\varphi\colon \mu_\infty(k)\to \mu_\infty(\CC)$. If $\chara(k)=p>0$, then let $W(k)$ be the Witt ring   of $k$ with field of fractions $K$. The reduction map $W(k)\to k$ induces a canonical isomorphism
$\mu_\infty'(K)\to \mu_\infty(k)$, where $\mu_\infty'$ denotes the group of roots of unity with order prime to $p$. We may then take an injective homomorphism $\varphi\colon \mu_\infty(k) = \mu_\infty'(K)\rightarrow \mu_\infty(\CC)$.}

Then the function $\phi_V:=\tr_B(\cdot; V)$ lies in $\Cl(G_\reg,\CC)$, the space of $\CC$-valued class functions on $G_\reg$. Extending by linearity, we may define the \emph{Brauer character} for any $\alpha\in R_k(G)$:
\[
\phi_\alpha=\tr_B(\cdot;\alpha)\colon G_\reg\rightarrow \CC.
\]
This induces an isomorphism of $\CC$-algebras (\cite[page 149, Theorem 42']{Ser77}):
\[
\Phi\colon R_k(G)_\CC= R_k(G)\otimes_\ZZ\CC \rightarrow \Cl(G_\reg, \CC), \quad \alpha\mapsto \phi_\alpha.
\]
In particular, $\alpha\in R_k(G)$ is uniquely determined by its Brauer character $\phi_\alpha$. 
For a subgroup $H\subset G$, $\alpha\in R_k(H)$, and $g\in G_\reg$, we have (cf.~\cite{Ser77}, Theorem~12 page 30 and Exercise 18.2, page 150]): 
\begin{equation}\label{eq: induced character}
\phi_{\Ind_H^G \alpha}(g) =
\begin{cases}
\displaystyle \frac{1}{|H|}\sum_{\text{\tiny$\begin{matrix}  t\in G\\  t^{-1}gt\in H \end{matrix}$}} \phi_\alpha(t^{-1}gt) & \text{if $\displaystyle g\in \bigcup_{t\in G} tHt^{-1}$},\\
0 & \text{otherwise}.
\end{cases}    
\end{equation}
In particular, taking $H=\{ 1 \}$, we see that for any $g\in G$, 
\begin{equation} \label{eq:trB_reg}
  \tr_B(g; k[G])=
  \begin{cases}
    |G| &  \text{if \ } g=1\\  
    0   & \text{otherwise}. 
  \end{cases}
\end{equation}
One may define a Hermtian form on $\Cl(G_\reg, \CC)$ by
\begin{equation}\label{eq: inner prod}
 \langle \phi_1, \phi_2\rangle := \frac{1}{|G|}\sum_{g\in G_\reg} \phi_1(g^{-1})\phi_2(g).   
\end{equation}
Then, by \cite[page 148, (vii)]{Ser77} or \cite[Proposition~10.2.1]{Web16}, for $k[G]$-modules $P$ and $V$, with $P$ projective, one has
\[
\dim_\CC \Hom_G(P, V) = \langle \phi_P, \phi_{V} \rangle.
\]
Via the isomorphism $\Phi\colon R_k(G)_\CC\rightarrow \Cl(G_\reg, \CC)$, we may define, for $\alpha, \beta \in R_k(G)_\CC$,
\[
\langle \alpha, \beta\rangle:=\langle \phi_\alpha,\phi_\beta\rangle.
\]

\begin{rmk}
If $k=\CC$ and $\varphi\colon \mu_\infty(k)\rightarrow\mu_\infty(\CC)$ is the identity map, then $\tr_B$ is just the usual trace and $\langle \alpha,\beta\rangle$ in \eqref{eq: inner prod} is an inner product of their virtual characters $\phi_\alpha$ and $\phi_\beta$.
\end{rmk}

Let $H\subset G$ be a subgroup. For $\alpha\in R_k(H)$ and $\beta\in R_k(G)$, one has, by the Frobenius reciprocity,
\begin{equation}\label{eq: Frobenius2}
\langle \Ind_H^G\alpha,\beta\rangle = \langle \alpha,\Res^G_H \beta\rangle.
\end{equation}

Suppose that $|G|$ is invertible in $k$. Then every $k[G]$-module is projective, and the classes of the irreducible $k[G]$-modules form an orthonormal basis for $R_k(G)_\QQ$ with respect to the bilinear form $\langle\cdot,\cdot\rangle$ (cf.~\cite[page 121]{Ser77} or \cite[Theorem~10.2.2]{Web16}). In this case, $\mu_V(\alpha)=\langle [V], \alpha\rangle$ for $V\in \Irr_k(G)$ and $\alpha\in R_k(G)_\QQ$. 

This is not true if $p$ divides $|G|$, as the following example shows.
\begin{ex}
Suppose that $\chara(k)=p>0$ and $G=\ZZ/p\ZZ$. Then $R_k(G)=\ZZ\cdot[1_G]$, and $\langle [1_G], [1_G]\rangle = \frac{1}{p}\in \QQ$.
\end{ex}

\medskip

Let us finish this recap of representation theory by giving some elementary results on induction and restriction of representations.

\begin{lem}  \label{lem:ind_res}
  Let $H\to \oH=H/N$ be a surjective group homomorphism.
  \begin{enumerate}
    \item Let $\oK$ be a
  subgroup of $\oH$ and let $K$ be its pre-image in $H$:
  \[
\begin{tikzcd}
K\arrow[r, hook]\arrow[d, two heads] & H\arrow[d, two heads] \\
\oK\arrow[r, hook] & \oH
 \end{tikzcd}
 \]
 Then for any
  $\oK$-module $M$ we have a canonical isomorphism of
  $H$-modules
  \[
\Res^{\oH}_H (\Ind^{\oH}_{\oK} M)\simeq \Ind^H_K (\Res^{\oK}_K M). 
\]
\item We have $\Res^{\oH}_H k[\oH] \simeq \Ind_N^H 1_N$ where
  $1_N$ corresponds to the unit representation of $N$.  
 \item Let $M$ be a $k[H]$-module. 
  Then we  have
\[
  \Ind_N^H\Res^H_N(M) \simeq k[H/N]\otimes_k M. 
  % M \oplus (M\otimes \chi). 
\]  
\end{enumerate}
\end{lem}

\begin{proof}  (1) By definition, 
  \[
    \Ind^H_K \Res^{\oK}_K M = k[H]\otimes_{k[K]}
    (k[\oK]\otimes_{k[\oK]} M)=(k[H]\otimes_{k[K]} k[\oK])\otimes_{k[\oK]} M
    \]
  and
  \[
\Res^{\oH}_H \Ind^{\oH}_{\oK} M = k[\oH]\otimes_{k[\oK]} M. 
\]
So it is enough to show that the canonical homomorphism
\[ k[H]\otimes_{k[K]} k[\oK]\to k[\oH] \]
is an isomorphism. This can be checked with a basis of $k[H]$ over
$k[K]$ and by using the canonical isomorphism $H/K\to \oH/\oK$.

(2)-(3) See \cite[\S 3.3, Example 1, (2) and (5)]{Ser77}.
\end{proof}

\begin{lem}\label{lem:MkG} Let $G$ be any finite group. 
  Let $M$ be a $k[G]$-module with $\dim_k M=r$. Then
  \[ M\otimes_k k[G] \simeq k[G]^{\oplus r}. \] 
\end{lem}

\begin{proof} 
Let $M^0$ be the $k$-vector space $M$ endowed with the trivial action of $G$. Define
  \[ \phi : M\otimes_k k[G] \to M^0 \otimes_k k[G] \simeq k[G]^{\oplus
      r}, \quad 
    x\otimes g \mapsto (g^{-1}x) \otimes g.
  \] 
 Then $\phi$ is a $G$-equivariant isomorphism.  
\end{proof}

\subsection{The \texorpdfstring{$G$}{}-equivariant Euler characteristic and reduction to normal connected varieties}
Let $X$ be a (possibly reducible) proper variety over $k$. Let $G$ be 
a finite group acting on $X$, and $\sF$ a $G$-equivariant sheaf of
$k$-vector spaces (see \cite[Section 39.12]{Stk25} for the
definition). Following \cite{EL80}, we define the
\emph{$G$-equivariant Euler characteristic} of $\sF$ to be the following element in the
Grothendieck group $R_k(G)$:
\begin{equation}\label{eq:chiGF} 
  \chi_G(\sF): = \sum_{q\geq 0} {(-1)^q} [H^q(X, \sF)].
\end{equation}
\begin{rmk}
$\chi_G(\sF)$ is additive in $\sF$, and $\chi_{G}$ defines a group homomorphism from the Grothendieck group $K_G(X)$ of $G$-equivariant sheaves of $k$-vector spaces  to $R_k(G)$.      
\end{rmk}
\begin{rmk}
    If $H$ is a subgroup of $G$, then any $G$-sheaf $\sF$ is also an $H$-sheaf. We have then $\Res^G_H \chi_G(\sF)=\chi_H(\sF)$. In particular, if $H=\{1\}$ is the trivial subgroup, then $\chi_{H}(\sF)$ is just the usual Euler characteristic, which is an integer.
\end{rmk}

\begin{rmk}
If the kernel $H$ of the action homomorphism $G\rightarrow \Aut(X)$ acts trivially on the cohomology groups $H^*(X,\sF)=\oplus_i H^i(X, \sF)$, which is the case for sheaves such as $\sO_X$ and $\CC_X$ that are intrinsic to $X$, then 
we may take the image $\oG$ of $G\rightarrow \Aut(X)$, and it is clear that
\begin{equation}\label{eq: G restr to oG}
\chi_{G}(\sF) = \Res^{\oG}_{G} \left(\chi_{\oG}(\sF)\right).
\end{equation}
Thus we may assume that $G$ is a finite subgroup of $\Aut(X)$ and the action of $G$ on $X$ is faithful. This we will do in Sections~\ref{sec: nodal} and \ref{sec: CW}.

In general, $H$ may acts nontrivially on $H^*(X,\sF)$ and it is more natural to include non-faithful actions of a group $G$ on $X$. Thus we do not assume the faithfulness of the action in Sections~\ref{sec: prelim} and \ref{sec: CW singular}, unless otherwise specified. Note that statements of previous works, including \cite{EL80} and \cite{Koc05}, assume the action to be faithful.
\end{rmk}

% Equivariant Euler characteristic  

We want to relate $\chi_G(\sF)$ to the
Euler characteristics of the pull-backs of $\sF$ to the connected
components of the normalization of $X$. 
\medskip

\subsubsection{Connected components.} We now explain how to express $\chi_G(\sF)$ in terms of the Euler
characteristics of the restrictions of $\sF$ to the connected
components. 
Let $X$ be a projective variety over $k$ endowed with the action of
a finite group $G$. 
Let $X_1, \dots, X_n$ be the connected components of $X$.
Disconnected varieties can occur as the normalization of connected
reducible varieties, for example. 
Let $\sF$ be a $G$-equivariant coherent sheaf on $X$. 
We want to describe the $k[G]$-modules
$H^* (X, \sF)$ in terms of the $H^*(X_i,
\sF|_{X_i})$'s. 

This fits into the following general situation. Let $M$ be a $k[G]$-module. Suppose that
  $M=\bigoplus_{ 1\leq i\leq n} M_i$ as $k$-vector spaces and that 
  for all $ i\leq n$ and for all $g\in G$, there exists $j
  =g*i$ such that
  $gM_i=M_j$. Thus $G$ acts on $[[1,n]]:=\{ 1,..., n\}$.  

\begin{lem}\label{lem:ind_connect}  Let $G_i=\{ g \in G \ | \ g(M_i)=M_i\}$ for
  any $i\leq n$. Then  
\[ [M]=\sum_{i\leq n} \frac{|G_i|}{|G|} \Ind^G_{G_i}[M_i] \in R_k(G)_\QQ. \] 
\end{lem} 

\begin{proof} Denote by $s\colon [[1, n]] \to [[1,n]]/G$ 
  the quotient map. For all 
  $\alpha\in [[1,n]]/G$ and all  $i\in s^{-1}(\alpha)$,
  we have a canonical isomorphism of $k[G]$-modules
  \[
    \Ind_{G_i}^G M_i =k[G]\otimes_{k[G_i]} M_i\to
    \bigoplus_{j\in  s^{-1}(\alpha)} M_j, \quad g\otimes x_i \mapsto gx_i \in
    M_{g*i}.     
  \]
  Taking direct sum, we get an isomorphism
  \[
    \bigoplus_{i\in s^{-1}(\alpha)} \Ind_{G_i}^G M_i \to
    \left(\bigoplus_{j\in  s^{-1}(\alpha)} M_j\right)^{|s^{-1}(\alpha)|}. 
  \]
  As $|s^{-1}(\alpha)|=|G|/|G_i|$, we have 
\[  \sum_{i\in s^{-1}(\alpha)} \frac{|G_i|}{|G|} \Ind_{G_i}^G [M_i] 
  = \left[\bigoplus_{j\in  s^{-1}(\alpha)} M_j\right] \in R_k(G)_\QQ. \] 
  Summing over all the $\alpha\in [[1, n]]/G$ gives the desired equality. 
\end{proof}

\begin{cor}\label{lem: conn decomp}
  Keep the above notation with $G_i=\{ g\in G \ |  \ g(X_i)=X_i \}$. 
  Then we have 
\[
  [H^q(X, \sF)] = \sum_{1\leq i\leq n} \frac{|G_i|}{|G|} 
  \Ind_{G_i}^G [H^q(X_i, \sF|_{X_i})] \in R_k(G)_\QQ   
\]
for all $q\geq 0$, and
\[
  \chi_G(\sF) = \sum_{1\leq i\leq n} \frac{|G_i|}{|G|} 
  \Ind_{G_i}^G \chi_{G_i}(\sF|_{X_i}).
\]
\end{cor}

\subsubsection{Normalization} Finally, if $\nu : \tX \to X$ is the normalization morphism, we can
relate $\chi_G(\sE)$ to $\chi_G(\nu^*\sE)$, where $\sE$ is any
$G$-equivariant coherent locally free sheaf on $X$, as follows. Consider the canonical exact sequence of $\sO_X$-modules 
\begin{equation}\label{eq: sesnX}
0\rightarrow \sO_X\rightarrow \nu_*\sO_{\tX}\rightarrow \sS\rightarrow 0    
\end{equation}
where $\Supp \sS$ is exactly the non-normal locus of $X$. Note that 
the action of $G$ on $X$ lifts to $\tX$, so that 
the exact sequence \eqref{eq: sesnX} is 
$G$-equivariant. 
Tensoring \eqref{eq: sesnX} with $\sE$, we obtain a $G$-equivariant exact sequence
\begin{equation}\label{eq:exactsnE}
  0\rightarrow %%\sK \rightarrow
  \sE\rightarrow \sE\otimes \nu_*\sO_{\tX}=\nu_*(\nu^*\sE)\rightarrow \sE\otimes \sS\rightarrow 0     
\end{equation}
%% where $\sK$ is defined to be the kernel of
%% $\sF\rightarrow \sF\otimes \nu_*\sO_{\tX}$.
By additivity of $\chi_{G}$, we obtain from \eqref{eq:exactsnE}
\begin{equation}\label{eq: chi_G F locally free}
\chi_{G}(\sE) = \chi_{G}(\nu^*\sE) - \chi_{G}(\sE\otimes \sS).    
\end{equation}
In other words, we can compute $\chi_{G}(\sE)$ in terms of
$\chi_G(\nu^*\sE)$ and $\chi_{G}(\sE\otimes \sS)$, which live on the
normalization and the non-normal locus of $X$ respectively. We will
specialize further to the case when $X$ is a nodal curve, so that $\chi_G(\nu^*\sE)$ and $\chi_{G}(\sE\otimes \sS)$ are either known or possible to write down explicitly.

Finally, we define the tameness of a $G$-action, which is automatically satisfied over a base field $k$ of characteristic 0, and which becomes a necessary requirement if one uses the Lefschetz fixed point formula over a positive characteristic field (\cite{Don69}).

\begin{defn}[Tameness] \label{defn:tame_action} 
Let $C$ be a proper reduced curve over $k$, and $G$ a finite group acting on $C$.  We say that the action of
$G$ on $C$ is \emph{tame}, or that $C$ is a \emph{tame $G$-curve}, if the characteristic exponent of $k$ does not divide the order of $G_Z$ for any irreducible subvariety $Z\subset C$. If $G$ is a subgroup of $\Aut(C)$, acting tamely on $C$, then we will say that $G$ is a \emph{tame subgroup of $\Aut(C)$.}
\end{defn} 

\section{The Chevalley--Weil formula for proper reduced curves}\label{CW_n_connected}  

\begin{setup}\label{setup: curve}
 Throughout  this section, we adopt the following notation.
 \begin{itemize}
     \item $C$ is a proper reduced curve over $k$, possibly disconnected;
     \item $C=\bigcup_{1\leq i\leq n}C_i$ is the irreducible decomposition of $C$;
     \item $\nu : \tC=\bigcup_{1\leq i\leq n}\tC_i \to C$ is the normalization map; 
     \item $G$ is a finite group  acting tamely on $C$ (Definition~\ref{defn:tame_action}); 
     \item $\pi \colon C\to D:=C/G$ is the quotient morphism. 
     \item For each $1\leq i\leq n$ (see the end of \S 1), 
\begin{equation}\label{eq: G_j}
  G_i:=G_{\{\tC_i\}}=G_{\{ C_i\}},
  \quad I_i:= G_{\tC_i}=G_{C_i}, \quad \oG_i :=
    G_i/I_i. 
\end{equation}
Note that $I_i$ is a normal subgroup of $G_i$, so $\oG_i$ is a group.
\item For any smooth point $P\in C$, denote by
  $G_P^\circ$ the group $I_i$ for the unique irreducible component
    $C_i$ of $C$   passing through $P$
  %% denotes the (normal) subgroup of $G_P$ that pointwise fixes a curve passing through $P$,
and by $\oG_P=G_P/G_P^\circ$ the quotient group.
\item $\sE$ is a locally free $G$-sheaf of rank $r$ on $C$.
 \end{itemize}
\end{setup} 
If the action of $G$ on $C$ is free, then
\cite[Theorem~2.4]{EL80}   already gives
\begin{equation}\label{eq:free} 
\chi_G(\sE) = \frac{1}{|G|}\chi(\sE)[k[G]] \in R_k(G)_\QQ. 
\end{equation}
In this section, we will give a formula
expressing $\chi_G(\sE)$ in
terms of a ``regular part'' which looks like
$\frac{1}{|G|}\chi(\sE)[k[G]]$,  plus a ``ramification part'' $\sum_{Z\subseteq C}\Gamma_G(\sE)_Z$
coming from the inertia loci $Z$ of the $G$-action; see Theorem~\ref{thm: CW singular}.

\subsection{The Chevalley--Weil formula for smooth curves}

Let notation be as in Set-up~\ref{setup: curve}, so we have a locally free $G$-sheaf $\sE$ of rank $r$ on a tame $G$-curve $C$. In this subsection, we assume that
the curve $C$ is smooth.

For a point $P\in C$, the Zariski cotangent space $T_P^*C$ of $C$ at $P$ is a 1-dimensional $\oG_P$-representation. We denote by $\theta_P$ the class of $T_P^*C$ in $R_k(G_P)$ (it is denoted by $\chi_P$ in \cite{Koc05}). It is helpful to notice that we have canonically
\begin{equation} \label{eq: theta_P}
 \theta_P=[\omega_{C,P}|_{P}]=[\sO_C(-P)|_P].
\end{equation}  
\if false THIS IS DONE IN \S 2 
  Following Nakajima \cite[\S 3]{Nak86a} we
can define the Brauer character of $\theta_P$ (\cite[\S 18.1]{Ser77}) as 
\[
\oG_P\to \mu'_{\infty}(k) \xrightarrow{\varphi}
\mu_{\infty}(\CC), \quad g\mapsto \varphi((g(t)t^{-1})(P)),  
\]
where $t$ is any uniformizing element of $\sO_{C,P}$ and $\varphi$ is a fixed injective homomorphism; it does not depend on the choice of $t$. 
\fi 
  
\begin{defn}[Ramification modules on a smooth $G$-curve]\label{defn: ramification module sm}
  Let $Z$ be an irreducible closed subvariety of $C$, so it is either a point or a connected component of $C$. Define an element $\Gamma_{G_Z}(\sE|_Z)\in R_k(G_Z)_\QQ$
as follows:
\begin{enumerate}[\rm (1)]
\item If $Z=C_i$ is an irreducible component of $C$, 
\begin{equation*}\label{eq:Gamma_i}
  \Gamma_{I_i}(\sE|_{C_i}) = \displaystyle \chi_{I_i}(\sE|_{C_i}) -
  \frac{\chi(\sE|_{C_i})}{|I_i|}[k[I_i]]. 
\end{equation*}
It should be noticed that $I_i$ acts trivially on $C_i$, but its action on
$\sE|_{C_i}$ may not be trivial. 
\item If $Z=P$ is a point of $C$, 
\begin{equation*}\label{eq: Gamma_P}
\Gamma_{G_P}(\sE|_P) = 
  \frac{1}{|\oG_P|}\left(\sum_{d=0}^{|\oG_P|-1}d\theta_P^d\right)\otimes[\sE|_P]
  \otimes\left(\frac{1}{|\oG_P|}\Ind_{G_P^\circ}^{G_P}[1_{G_P^\circ}]-[1_{G_P}]\right).
\end{equation*}
\end{enumerate}
We define the \emph{ramification module} of $\sE$ at $Z$ in both cases
to be 
\begin{equation}\label{eq: Gamma_G induced from Z}
\Gamma_G(\sE)_Z:=\frac{|G_{ Z}|}{|G|}\Ind_{G_Z}^G \Gamma_{G_Z}(\sE|_Z).
\end{equation}
This is an element of virtual degree $0$ in $R_k(G)_\QQ$.
\end{defn}
  
The ramification modules are so defined as to reflect the local contributions of fixed locus of the $G$-action given by the Atiyah--Singer's holomorphic Lefschetz fixed point formula (\cite{AS68, Don69}).

\begin{lem}\label{lem:gp_GG} The following properties hold true. 
\begin{enumerate}[\rm (1)] 
\item We have $\Gamma_G(\sE)_{C_i}=0$ if $I_i$ is trivial.
  For $P\in C$, we have $\Gamma_G(\sE)_{P}=0$ if
$\overline{G}_{P}$ is trivial. 
  In particular, $\Gamma_G(\sE)_Z\neq 0$ only for finitely many $Z$'s.  
    \item For any $g\in G$, we have $\Gamma_G(\sE)_Z = \Gamma_G(\sE)_{g(Z)}$.
\end{enumerate}   
\end{lem}

\begin{proof} (1) follows from the definition of $\Gamma_G(\sE)_{Z}$.
\medskip

(2) Let $\rho_g\colon G\rightarrow G, \, t\mapsto g^{-1}tg$ be the inner automorphism of $G$ induced by $g$. Then $\rho_g(G_{g(Z)})=G_{Z}$, and we use the same $\rho_g$ to denote the induced isomorphism $G_{g(Z)}\rightarrow G_Z$. Since $\sE$ is $G$-equivariant, we have $g^*(\sE|_{g(Z)}) = \sE|_{Z}$, and 
\[
\rho_g^* \Gamma_{G_Z}(\sE|_Z) = \Gamma_{G_{g(Z)}}(\sE|_{g(Z)}) 
\]
It follows that
\begin{multline*}
  \Gamma_G(\sE)_Z = \frac{|G_Z|}{|G|}\Ind_{G_Z}^G \Gamma_{G_Z}(\sE|_Z) =
  \frac{G_Z}{|G|} k[G]\otimes_{k[G_Z]}\Gamma_{G_Z}(\sE|_Z) \\
  =\frac{|G_Z|}{|G|} \rho^*k[G]\otimes_{\rho^*k[G_Z]}\rho^*\Gamma_{G_Z}(\sE|_Z)
  =\frac{|G_Z|}{|G|} k[G]\otimes_{k[G_{g(Z)}]}\Gamma_{G_{g(Z)}}(\sE|_{g(Z)}) \\
  = \Gamma_G(\sE)_{g(Z)}
\end{multline*}
\end{proof}

\begin{lem}\label{lem: tr G_Po vanish} Let $g_0\in G_{\reg}\setminus \{ 1 \}$ 
  (see \eqref{eq: G_reg} for the definition). Denote by $C^{g_0}$ the
  fix locus of $g_0$ acting on $C$. 
\begin{enumerate}[\rm (1)] 
\item Let $Z$ be a closed irreducible subvariety of $C$. Then 
\[
\tr_B(g_0; \Gamma_G(\sE)_{Z}) = \frac{|G_{\{Z\}}|}{|G|}
      \sum_{Z'} \tr_B(g_0;\Gamma_{G_{Z'}}(\sE|_{Z'}))
    \]
    where $Z'$ runs through the irreducible components of the orbit $G\cdot Z$ contained in $C^{g_0}$. In particular, 
$\tr_B(g_0;\Gamma_G(\sE)_Z)=0$ if $Z$ is not contained in the $G$-orbit of $C^{g_0}$. 
\item Let $P\in C^{g_0}$, then $\tr_B(g_0;\Gamma_{G_P}(\sE|_P))=0$ if
  $P$ is contained in an irreducible component of dimension $1$ of
  $C^{g_0}$ (equivalently, $g_0\in G_P^{\circ}$). Otherwise ({\it i.e.}, $P$ is an isolated point of $C^{g_0}$) we have 
\[
\tr_B(g_0;\Gamma_{G_P}(\sE|_P))=
\displaystyle \frac{\tr_B(g_0;\sE|_P)}{1-\tr_B(g_0; \theta_P)}
\]
\end{enumerate}
\end{lem}

\begin{proof}
(1) For any irreducible closed subset $Z$ of $C$, we have by \eqref{eq: Gamma_G induced from Z}, 
\begin{equation}\label{eq:tr_g_0}
  \tr_B(g_0; \Gamma_G(\sE)_Z) =
  \displaystyle \frac{1}{|G|}\sum_{\text{\tiny$\begin{matrix} g\in G\\ g^{-1}g_0g\in G_Z \end{matrix}$}} \tr_B(g^{-1}g_0g; \Gamma_{G_Z}(\sE|_Z)), 
\end{equation}
the sum being trivial if there is no $g\in G$ such that
$g^{-1}g_0g\in G_Z$. The latter condition is equivalent to 
$g_0\in gG_Zg^{-1}=G_{gZ}$, or, $gZ\subseteq C^{g_0}$. As
\[ \tr_B(g^{-1}g_0g; \Gamma_{G_Z}(\sE|_Z))=
  \tr_B(g_0; \Gamma_{G_{gZ}}(\sE|_{gZ})),
\]
we have 
\begin{equation}
  \begin{split} 
  \tr_B(g_0; \Gamma_G(\sE)_Z) & 
    = \frac{1}{|G|}\sum_\text{\tiny$\begin{matrix} g\in G\\ gZ \subseteq C^{g_0} \end{matrix}$}
    \tr_B(g_0; \Gamma_{G_{gZ}}(\sE|_{gZ})) \\
  & = \frac{|G_{\{ Z\}}| }{|G|} \sum_{Z'} \tr_B(g_0;\Gamma_{G_{Z'}}(\sE|_{Z'}))
    \end{split} 
\end{equation}
where $Z'$ runs through the set $\{ gZ  \ | \ g\in G, \ gZ \subseteq C^{g_0} \}$.
\medskip

(2) By the definition of $\Gamma_{G_P}(\sE|_P)$, we have
\begin{multline}\label{eq: tr G_P}
  \tr_B(g_0;\Gamma_{G_P}(\sE|_P)) =\\
  \frac{1}{|\oG_P|}\sum_{d=0}^{|\oG_P|-1}\tr_B\left(g_0;d\theta_P^d\otimes[\sE|_P]\right)
  \cdot\tr_B\left(g_0;\frac{1}{|\oG_P|}\Ind_{G_P^\circ}^{G_P}[1_{G_P^\circ}]
  -[1_{G_P}]\right). 
\end{multline}
If $g_0\in G_P^\circ$ (so the component $C_i$ of $C$ containing $P$ is contained in
$C^{g_0}$), then
\[
\tr_B\left(g_0;\frac{1}{|\oG_P|}\Ind_{G_P^\circ}^{G_P}[1_{G_P^\circ}]-[1_{G_P}]\right) = 0,  
\]
hence $\tr_B(g_0;\Gamma_{G_P}(\sE|_P))=0$. If $g_0\in G_P\setminus G_P^\circ$, then 
\[
\tr_B\left(g_0;\frac{1}{|\oG_P|}\Ind_{G_P^\circ}^{G_P}[1_{G_P^\circ}]-[1_{G_P}]\right) = -1.
\]
As $\tr_B(g_0;\theta_P)\in \mu'_\infty(\CC) \setminus \{ 1 \}$ (because
the image of $g_0$ in $\Aut(\mathcal{O}_{C, P})$ is non trivial),
we have by \cite[Lemma~1.2]{Koc05} 
\[
\sum_{d=0}^{|\oG_P|-1}\tr_B(g_0;d\theta_P^d) = \frac{|\oG_P|}{\tr_B(g_0; \theta_P)-1}
\]
Plugging this into Equality \eqref{eq: tr G_P}, we obtain the desired equality.
\end{proof} 

\if false
  
\begin{equation}\label{eq: tr g_0}
    \begin{split}
&\tr_B(g_0; \Gamma_G(\sE)_Z) = \tr_B\left(g_0; \frac{|G_Z|}{|G|}\Ind_{G_Z}^G\Gamma_{G_Z}(\sE|_Z)\right) \\
&  =
  \begin{cases}
\displaystyle \frac{1}{|G|}\sum_{\text{\tiny$\begin{matrix} g\in G\\ g^{-1}g_0g\in G_Z \end{matrix}$}} \tr_B(g^{-1}g_0g; \Gamma_{G_Z}(\sE|_Z)) & \text{if $\displaystyle g_0\in \bigcup_{g\in G} gG_Zg^{-1}$},\\
 0 & \text{otherwise}
 \end{cases} 
    \end{split}
\end{equation}
Since $G_{g(Z)}=g\cdot G_Z\cdot g^{-1}$, we have 
\[
\bigcup_{g\in G} gG_Zg^{-1} = \bigcup_{g\in G} G_{g(Z)}.
\]
By assumption, $g_0\notin \cup_{g\in G} G_{g(Z)}$, we infer that $\tr_B(g_0;\Gamma_G(\sE)_Z)=0$.

\medskip

(4) For $g\in G$, we have the following implications:
\begin{equation}\label{eq: tr g_00}
\begin{split}
g^{-1}g_0g\in G_Z &\Longleftrightarrow g_0\in gG_Zg^{-1} = G_{g(Z)} \\
& \Longrightarrow  \tr_B(g^{-1}g_0g; \Gamma_{G_Z}(\sE|_Z)) = \tr_B(g_0, \Gamma_{G_{g(Z)}}(\sE|_{g(Z)}))
\end{split}
\end{equation}
It follows from \eqref{eq: tr g_00} that, if $g_0\in \bigcup_{g\in G} gG_Zg^{-1} $, then
\begin{equation}\label{eq: tr g_01}
    \begin{split}
       \tr_B(g_0; \Gamma_G(\sE)_Z) &= \frac{ |G_Z| }{|G|}\sum_{\text{\tiny$\begin{matrix}g\in R\\ g_0\in G_{g(Z)}\end{matrix}$}}\tr_B(g_0; \Gamma_{G_{g(Z)}}(\sE|_{g(Z)})) \\
       &= \frac{|G_Z| }{|G|} \sum_{\text{\tiny$\begin{matrix} Z'\subset G\cdot Z \\ g_0\in G_{Z'} \end{matrix}$}} \tr_B(g_0;\Gamma_{G_{Z'}}(\sE|_{Z'}))
    \end{split}
\end{equation}
where $R\subset G$ is a set of representatives for the coset $G/G_Z$, and   $Z'$ runs through the irreducible components of the orbit $G\cdot Z$ such that $g_0\in G_{Z'}$.

\medskip

\end{proof}
\fi

% \begin{rmk}
% For $P\in C$, set 
% \[
% \Psi_P=\frac{|\oG_P|-1}{2}[k[\oG_P]]-\sum_{d=0}^{|\oG_P|-1}d\theta_P^d\in R_k(\oG_P)_\QQ.
% \]
% Then, using Lemma~\ref{lem:MkG}, one has
% \[
% \Gamma_G(\sE)_P =\frac{|I_i|}{|G|}\Ind_{G_P}^G\left(\Psi_P\otimes [\sE|_P]\right).
% \]
% \end{rmk}

%  \begin{rmk}
%  There are other ways to write down the expression of
%    $\Psi_P$: 
% \[   \Psi_P=\sum_{d=0}^{|\oG_P|-1}
%   d\left(\frac{[k[\oG_P]]}{|\oG_P|}-\theta_P^d\right) \]
% and also 
% \[   \Psi_P=\sum_{d=0}^{|\oG_P|-1} \left(\frac{|\oG_P|-1}{2}-d\right) \theta_P^d. 
% \]
% Which expression is better depends on the application circumstances.
% \end{rmk}

% \begin{rmk} For a subvariety $Z\subset C$, which is either a single point or a connected component of $C$, we have  

% For any subvariety $W\subset D$,  define 
% \[
% \Gamma_G(\sE)_W:=\sum_{Z\subset \pi^{-1}(W)} \Gamma_G(\sE)_P.
% \]
% Since any two subvarieties $Z, Z'$ of $C$ in the same $G$-orbit have the same ramification
%   $\Gamma_G(\sE)_P= \Gamma_G(\sE)_{P'}$, and $|\pi^{-1}(Q)|=\frac{|G|}{|G_P|}$ for any $P\in \pi^{-1}(Q)$, we obtain
%   \[
%   |G_P|\Gamma_G(\sE)_Q =|G|\Gamma_G(\sE)_P.
%   \]
% \end{rmk}

\if false 
\begin{lem}\label{lem: Gamma_G sm}
Let $g_0\in G_{\reg}$ be $(\chara\,k)$-regular (see \eqref{eq: G_reg} for definition), and $Z\subset C$ a closed irreducible subvariety, which is either a single point $P$ or a connected component of $C$. Then the following holds:
\begin{enumerate}
    \item $\Gamma_G(\sE)_{Z}=0$ if $G_Z$ is trivial. In particular, $\Gamma_G(\sE)_Z\neq 0$ only for finitely many $Z$'s.  
    \item For any $g\in G$, we have $\Gamma_G(\sE)_Z = \Gamma_G(\sE)_{g(Z)}$.
    \item If $g_0\notin\cup_{g\in G} G_{g(Z)}$, or equivalently, if $Z$ is not contained in the $G$-orbit of $C^{g_0}$, then $\tr_B(g_0;\Gamma_G(\sE)_Z)=0$.
    \item If $g_0\in \cup_{g\in G} G_{g(Z)}$, or equivalently, if $Z$ is contained in the $G$-orbit of $C^{g_0}$, then 
    \[
     \tr_B(g_0; \Gamma_G(\sE)_{Z}) = \frac{|G_Z|}{|G|}\sum_{\text{\tiny$\begin{matrix} Z'\subset G\cdot Z \\ g_0\in G_{Z'} \end{matrix}$}} \tr_B(g_0;\Gamma_{G_{Z'}}(\sE|_{Z'}))
    \]
    where $Z'$ runs through the irreducible components of the orbit $G\cdot Z$.
    \item If $Z$ is not a connected component of the orbit $G\cdot C^{g_0}$, then we have $\tr_B(g_0;\Gamma_G(\sE)_Z)=0$.
\end{enumerate}
\end{lem}
\begin{proof}

(5) By (4), we may assume that $Z$ is contained in the $G$-orbit of $C^{g_0}$. By assumption, $Z$ is not a component of the $G$-orbit of $C^{g_0}$, so $Z$ must be a point and $C^{g_0}$ contains a curve, say $C_i$ passing through a point of $G\cdot Z$. Therefore, for any $Z'\in G\cdot Z$ such that $g_0\in G_{Z'}$, we have $g_0\in G_{Z'}^\circ$. By Lemma~\ref{lem: tr G_Po vanish}, the summands $\tr_B(g_0;\Gamma_{G_{Z'}}(\sE|_{Z'}))$ in \eqref{eq: tr g_01} vanish, and so does their sum $ \tr_B(g_0; \Gamma_G(\sE)_Z)$.
\end{proof}
\fi

The following statement is a 
generalization of \cite[Theorem~1]{Koc05} 
to the case of possibly non-faithful action of a finite group on a
possibly disconnected smooth curve.  The proof is similar to
that of \cite{Koc05}, using the algebraic version of the Lefschetz fixed point theorem given by \cite{Don69}. 

\begin{thm}\label{thm: sm}
  Let $C$ be a (non necessarily connected) smooth projective curve over
  $k$,  and $G$ a finite group acting tamely on $C$.
Let $\sE$ be a $G$-equivariant
locally free sheaf of rank $r$. Then we have the following equality
    in $R_k(G)_\QQ$:  
\begin{equation}\label{eq: CW sm}
\chi_G(\sE) = \frac{1}{|G|}\chi(\sE)[k[G]]+ \sum_{Z\subseteq C}\Gamma_G(\sE)_Z 
\end{equation}
where $Z$ runs through the irreducible closed subvarieties of $C$.
\end{thm}

\begin{proof}
By \cite[page 149, Corollary 1]{Ser77} or the results recalled in Section~\ref{sec: character}, it suffices to show that the Brauer trace of any $g\in G_\reg$ on the both sides of \eqref{eq: CW sm} are equal (see \eqref{eq: Brauer trace} for the definition of the Brauer trace). 

For $g=1$, the Brauer trace on the left-hand side of \eqref{eq: CW sm}
\[
\tr_B(1; \chi_G(\sE)) = \dim_k H^0(C,\sE) - \dim_k H^1(C,\sE) = \chi(\sE).
\]
which is equal to the trace on the right-hand side, computed as follows:
\[
\tr_B\left(1; \frac{1}{|G|}\chi(\sE)[k[G]]+ \sum_{Z\subseteq C}\Gamma_G(\sE)_Z\right) =  \tr_B\left(1; \frac{1}{|G|}\chi(\sE)[k[G]]\right)=\chi(\sE).
\]
 
If $g\neq 1$, then by the Lefschetz fixed point theorem (\cite[(4.6)]{AS68} for
$k=\CC$ and \cite[Corollary 5.5]{Don69} for any algebraically closed field $k$),
we have
\begin{equation}\label{eq: CW sm Lefschetz}
\begin{split}
\tr_B(g; \chi_G(\sE))  = \sum_{P\in Z_0(C^g)} \frac{\tr_B(g;[\sE|_P])}{1-\tr_B(g; \theta_P)} + \sum_{E\in \mathrm{Z}_1(C^g)} \sum_{d=0}^{|g|-1}\tr_B(g;\phi^d)\int_{E}\ch(\sE|_{E,d})\td(E)
\end{split}
\end{equation}
where the notation in the above formula is as follows:
\begin{itemize}
\item $Z_0(C^g)$ is the set of the isolated points in the fixed locus $C^g$ of $g$; 
\item $\mathrm{Z}_1(C^g)$ is the set of the 1-dimensional components of $C^g$;
    \item for $0\leq d\leq |g|-1$, ${\sE|_{E,d}} $ denotes the eigensheaf of the $\langle g\rangle$-sheaf $\sE|_E$ corresponding to a character $\phi^d\colon \langle g\rangle\rightarrow \mu_\infty(k), g\mapsto \xi^d$, where $\xi\in\mu_\infty(k)$ is a primitive $|g|$-th root of unity;
    \item $\td(E)\in \CH^*(E)$ denotes the Todd class of the curve $E$.
%    \item $f\colon E\rightarrow \mathrm{Spec}(k)$ denote the structure map. 
\end{itemize}
% We compute each summand in \eqref{eq: CW sm Lefschetz}: For $P\in C^g$, we have, by \cite[Lemma~1.2]{Koc05},
% \begin{equation}\label{eq: trace P}
% \frac{\tr_B(g;\sE|_P)}{1-\tr_B(g; T_P^*C)} = -\frac{1}{|\oG_P|}\sum_{d=0}^{|\oG_P|-1}d\tr_B(g;\theta_P)^d\cdot \tr_B(g;\sE|_P).
% \end{equation}
% For a $1$-dimensional component $E$ of $C^g$, we have, by the Hirzebruch--Riemann--Roch theorem,
% \begin{equation}\label{eq: trace E}
% \tr_B(g; f_! \ch(\sE|_E)\td(E))  = \tr_B(g; \chi_{\langle g\rangle}(\sE|_E)).    
% \end{equation}

Now we compute the trace of $g$ on the right-hand side of \eqref{eq: CW sm} for
a given $g\ne 1$. Then $\tr_B(g; k[G])=0$ (see \eqref{eq:trB_reg}), so 
\begin{equation}\label{eq: tr_B Gamma_G}
\begin{split}
\tr_B\left(g; \frac{1}{|G|}\chi(\sE)[k[G]]+ \sum_{Z\subseteq C}\Gamma_G(\sE)_Z\right) & = \tr_B\left(g; \sum_{Z\subseteq C}\Gamma_G(\sE)_Z\right).      
\end{split}
\end{equation}
We split the sum $\sum_{Z} \Gamma(\sE)_Z$ into the sum on the $\dim Z=0$ part
and the $\dim Z=1$ part.

By Lemma~\ref{lem: tr G_Po vanish}(1), % Part (2) and then Part (1), 
\[
  \tr_B(g; \Gamma_G(\sE)_P)=\frac{|G_P|}{|G|}\sum_{Q\in C^g\cap GP}
  \tr_B(g; \Gamma_{G_Q}(\sE|_{Q})) %= \frac{|G_P|}{|G|}\sum_{Q\in Z_0(C^g)\cap GP} \Psi_{g, Q}, 
\]
the right-hand side being zero if $C^g\cap GP=\emptyset$, that is, if
$P\notin GC^g$. 
Therefore, as $|G_Q|=|G_P|$ for $Q\in GP$, we have  
\[
  \begin{split} 
  \sum_{P\in C}  \tr_B(g; \Gamma_G(\sE)_P) = & \sum_
  \text{\tiny$\begin{matrix} P\in GC^g\\ Q\in C^g\cap GP \end{matrix}$} \frac{|G_P|}{|G|} \tr_B(g; \Gamma_{G_Q}(\sE|_{Q})) \\ 
    =& \sum_{Q\in C^g}\sum_{P\in GQ} \frac{|G_Q|}{|G|} \tr_B(g; \Gamma_{G_Q}(\sE|_{Q}))\\
    = & \sum_{Q\in C^g} \tr_B(g; \Gamma_{G_Q}(\sE|_{Q})). 
\end{split}    
\]
By Lemma~\ref{lem: tr G_Po vanish}(2),
\[
    \sum_{P\in C}  \tr_B(g; \Gamma_G(\sE)_P) =
\sum_{Q\in F^g} \frac{\tr_B(g;[\sE|_P])}{1-\tr_B(g; \theta_P)} 
\]
which is the first term of \eqref{eq: CW sm Lefschetz}.

\if false 
the Brauer trace $\tr_B\left(g; \Gamma_G(\sE)_Z\right)$ is nonzero only if $Z$ is a connected component of $G\cdot C^g$. Furthermore, using  Lemma~\ref{lem: Gamma_G sm} (4), we may write
\[
\eqref{eq: tr_B Gamma_G} = \sum_{Z\subset C^g} \frac{|G_Z|}{|G|}\tr_B(g; \Gamma_{G_Z}(\sE|_Z))
\]
where $Z$ runs through the connected components of $C^g$. If $Z=P$ is an isolated point of $C^g$, then by Lemma~\ref{lem: tr G_Po vanish}
\[
\tr_B(g; \Gamma_{G_P}(\sE|_P)) = \frac{\tr_B(g;[\sE|_P])}{1-\tr_B(g; [\theta_P])}
\]
\fi 

Reasonning in a similar way for closed irreducible subsets $Z\subseteq C$ of dimension $1$, we get 
\if false
  \[
\tr_B(g; \Gamma_G(\sE)_Z)=
\frac{|G_{\{ Z \}}|}{|G|}\sum_{Z'} \tr_B(g; \Gamma_{G_{Z'}}(\sE|_{Z'}))=
  \sum_{Z'} \frac{|G_{\{ Z' \}}|}{|G|}\tr_B(g; \Gamma_{G_{Z'}}(\sE|_{Z'}))
\]
where $Z'$ runs through the components of $GZ$ contained in $C^g$.
So the sum on the irreducible one-dimensional $Z$'s
\fi 
\[
\begin{split} 
  \sum_{Z\subset C} \tr_B(g; \Gamma_G(\sE)_Z)= & 
  \sum_Z \sum_
   \text{\tiny$\begin{matrix} E\in \mathrm{Z}_1(C^g)\\ E\subseteq GZ \end{matrix}$} \frac{|G_{\{ E \}}|}{|G|}\tr_B(g; \Gamma_{G_{E}}(\sE|_{E})) \\ 
  = & \sum_{E\in \mathrm{Z}_1(C^g)} \sum_{Z \subseteq GE}  \frac{|G_{\{ E \}}|}{|G|}\tr_B(g; \Gamma_{G_{E}}(\sE|_{E})) \\
  = & \sum_{E\in \mathrm{Z}_1(C^g)} \tr_B(g; \Gamma_{G_{E}}(\sE|_{E})) 
\end{split} 
\]
because $E$ is a component of $GZ$ if and only if $Z$ is a component of $GE$.
For $E\in \mathrm{Z}_1(C^g)$, we have $g\in G_E$,
$\tr_B(g, k[G_E])=0$ (see \eqref{eq:trB_reg}) and, 
by Definition~\ref{defn: ramification module sm},  
\[
\tr_B(g; \Gamma_{G_E}(\sE|_E)) = \tr_B\left(g; \chi_{G_E}(\sE|_{E}) - \frac{\chi(\sE|_{E})}{|G_E|}[k[G_E]]\right) 
= \tr_B\left(g; \chi_{G_E}(\sE|_{E})\right)
\]
with $G_E$ acting trivially on $E$. By Hirzebruch--Riemann--Roch theorem, 
\[
\tr_B\left(g; \chi_{G_E}(\sE|_{E})\right) =\sum_{d=0}^{|g|-1}\tr_B\left(g; \chi_{G_E}(\sE|_{E,d})\right) =  \sum_{d=0}^{|g|-1}\tr_B(g;\phi^d)\int_{E}\ch(\sE|_{E,d})\td(E).
\]

In conclusion, we have shown that traces of any $g\in G_\reg$ on the right hands of \eqref{eq: CW sm Lefschetz} and \eqref{eq: tr_B Gamma_G} coincide, and the theorem follows.
\end{proof}

\begin{cor}\label{cor: sm}
  Keep the hypothesis of Theorem~\ref{thm: sm} and notation of
  Set-up~\ref{setup: curve}. Then we have the following equality
  in $R_k(G)_\QQ$:  
\begin{equation}\label{eq: CW sm 2}
\chi_G(\sE) = \sum_{1\leq i\leq n}\frac{|I_i|}{|G|} \Ind_{I_i}^G
\chi_{I_i}(\sE|_{C_i})+ \sum_{P\in C}\Gamma_G(\sE)_P. 
\end{equation}
%%where $P$ runs through the points of $C$ such that the image of $G_P\rightarrow \Aut(\sO_{C,P})$ is nontrivial.
\end{cor}
\begin{proof}
In view of Theorem~\ref{thm: sm}, it suffices to notice that $\chi(\sE) = \sum_{1\leq i\leq n}\chi(\sE|_{C_i})$, and that
\begin{align*}
\frac{1}{|G|}\chi(\sE)[k[G]]+ \sum_{1\leq i\leq n}\Gamma_G(\sE)_{C_i} &=\frac{1}{|G|} \sum_{1\leq i\leq n}\chi(\sE|_{C_i})[k[G]]+\sum_{1\leq i\leq n}\Gamma_G(\sE)_{C_i} \\
&=\sum_{1\leq i\leq n}\frac{|I_i|}{|G|} \Ind_{I_i}^G \chi_{I_i}(\sE|_{C_i}).
\end{align*}
The last equality follows from Definition~\ref{defn: ramification module sm}.
\end{proof}

\begin{rmk} For the computation of $\Gamma_G(\sE)_P$, we also use the
  following  equality for $P\in C_i\subseteq C$
  (obtained by applying Lemma~\ref{lem:MkG} to $[k[\oG_P]]$) 
  \begin{equation} \label{eq:Gamma_G(E)_P}
    \Gamma_G(\sE)_P=\frac{|I_i|}{|G|}\Ind_{G_P}^G (\Psi_P\otimes [\sE|_P])  
  \end{equation}
  where
  \[ \Psi_P= \frac{|\oG_P|-1}{2}[k[\oG_P]] -\sum_{d=0}^{|\oG_P|-1} d\theta_P^d.\]   
\end{rmk} 

\begin{rmk} 
For comparison, we show how, if $C$ is connected and the action of $G$ on $C$ is faithful and tame, Theorem~\ref{thm: sm} recovers \cite[Theorem~1]{Koc05}. In this case, $\Gamma_G(\sE)_C=0$.
%% Let $D=C/G$.
By Riemann--Hurwitz's theorem we have   
$|G|\chi(\sO_D)=\chi(\sO_C)+\frac{1}{2}\deg R$, where
$R=\sum_{P\in C} (|G_P|-1)[P]$ is the 
ramification divisor of the quotient map $\pi\colon C\to D$.  As
$\chi(\sE)=r\chi(\sO_C)+\deg \sE$, where $r=\rk \sE$, 
by Theorem~\ref{thm: sm}, $\chi_G(\sE)$ is equal to    
\begin{align*} 
  % \chi_G(\sE)& =
                 \frac{1}{|G|}\chi(\sE) [k[G]] +
  \sum_{P\in C} \frac{1}{|G|}\Ind_{G_P}^{G}\left(r\frac{|G_P|-1}{2}[k[G_P]]-
  \sum_{d=0}^{|G_P|-1} d\theta_P^d \otimes [\sE|_P]\right)\\
  =\left(r\chi(\sO_D)+\frac{1}{|G|}\deg\sE\right)[k[G]] - \frac{1}{|G|}    \sum_{P\in C} \Ind_{G_P}^{G}\left( \sum_{d=0}^{|G_P|-1} d\theta_P^d \otimes [\sE|_P]\right). 
\end{align*}
which is \cite[Theorem~1]{Koc05}.
\end{rmk}

\begin{ex} Let $C$ be a
  hyperelliptic curve of genus $g$ over $k$ of $\chara(k)\ne 2$.
  Let $G\subseteq \Aut(C)$ be generated by the hyperelliptic
  involution. Let $P_0\in C$ be a 
  Weierstrass point, $q\in \ZZ$ and $\sE=\omega_C(q P_0)$.
  Denote by $\tau$ the nontrivial character $G\to \{ \pm 1 \}$.
  Then $[\sE|_{P_0}]=\theta_{P_0}\otimes \theta_{P_0}^{-q}=\tau^{1-q}$
  and $[\sE|_P]=\theta_P=\tau$ if $P\ne P_0$.  
For a Weierstrass point $P$, we have
\[
\Gamma_G(\sE)_{P}=\frac{1}{2}\left(\frac{1}{2}[k[G]]-\tau\right)\otimes [\sE|_{P}] = 
\begin{cases}
 \frac{1}{4}[k[G]]-\frac{1}{2}[1_G] & \text{if $P\neq P_0$} \\
  \frac{1}{4}[k[G]]-\frac{1}{2}\tau^{q} &\text{if $P= P_0$} 
\end{cases}
\]
Therefore
  \begin{align*}
    \chi_G(\omega_C(q P_0))&=\frac{1}{2}\chi(\omega_C(qP_0))[k[G]] + \sum_{P\in C^G} \Gamma_G(\sE)_{P} \\ &= \frac{1}{2}(g+q-1)[k[G]] + (2g+1)\left( \frac{1}{4}[k[G]]-\frac{1}{2}[1_G]\right) + \left(\frac{1}{4}[k[G]]-\frac{1}{2}\tau^{q}\right) \\
    &=\frac{(q-1)}{2}[k[G]]+ g\tau + 
    \frac{1}{2}(\tau-\tau^q) \\
   &=\lfloor\frac{q-1}{2}\rfloor[k[G]]+(g+\epsilon)\tau       
  \end{align*}
  with $\epsilon=1$ if $q$ is even, $\epsilon=0$ otherwise.
  See \cite[Proposition 3]{Kan86} for much more general results. 
\end{ex}
  
\begin{ex} \label{ex:etale_nonfree} Let $G$ be a finite
  group acting faithfully on a quasi-projective variety $X$ over $k$. When $X$ is
  connected, the quotient morphism $X\to X/G$ is \'etale if and only
  if $G$ acts freely on $X$. But this is no longer true when $X$ is not
  connected. Below we construct an example with $X\to X/G$ \'etale, $X/G$
  smooth and connected, but all points of $X$  have nontrivial stabilizer.
  %% In this case Equality~\eqref{eq:free} can fail to be true.
  % consider the following example in $\chara(k)=0$.
  Let $G$ be a finite group having a nontrivial subgroup $I$
such that $\cap_{g\in G} gIg^{-1}=\{ 1 \}$ (e.g. $G=S_3$ and $I$ any subgroup of order $2$). Let $Y$ be a
  projective smooth connected variety over $k$ and let
  \[ X=\Ind_I^G Y:=\{(y, a) \in Y \times (G/I) \}\]
  endowed with the action of $G$ defined by
  $g*(y, a)=(y, g*a)$, where $g*a$ is the natural action (on the left)
  of $G$ on the left cosets $G/I$. This action is faithful by the
   hypothesis on $I$.
  %% The canonical map $Y\to \Ind_H^GY/G$ is clearly an isomorphism.
  The quotient morphism $X\to Y$ is \'etale 
because it is an isomorphism when restricted to any connected
component $X_a:= Y \times \{ a \}$ of $X$. Finally if $a=gI$, then
the inertia and decomposition groups of $X_a$ are $I_a:=gIg^{-1}\ne \{ 1 \}$.  

For any $G$-equivariant coherent sheaf $\sE$ on $X$ 
such that $I_a$ acts trivially on $\sE|_{X_a}$ for $a\in G/I$,
we have
\[
\chi_G(\sE)=\frac{1}{|G|} \sum_{a\in G/H} \chi(\sE|_{X_a})
|I_a|\Ind_{I_a}^G{1_{I_a}}=\frac{1}{|G|}\chi(\sE) (|I|\Ind_{I}^G [1_{I}]). 
\]
But $|I|\Ind_{H}^G 1_{I}\ne [k[G]]$ because
their characters are different.
%% (consider the trace of any element of $I\setminus \{1\}$).
Therefore, the isomorphism~\eqref{eq:free} is 
false as soon as $\chi(\sE)\ne 0$. In particular,
\cite[Theoreom 2.4]{EL80} as stated is incorrect; it requires the freeness of the $G$-action, in order to be valid.  
\end{ex}

We compute the virtual multiplicity of $[1_G]$ in the $G$-equivariant Euler
characteristic of a pluri-log-canonical sheaf on a log smooth
$G$-curve.

\begin{prop}\label{prop: log smooth}
  Let $C$ be as in Theorem~\ref{thm: sm} and $G$  a tame finite subgroup of $\Aut(C)$. Let $T$ be a
$G$-invariant finite subset of $C$, also viewed as a reduced divisor on $C$.  Let
$\oT=\pi(T)$, where $\pi\colon C\rightarrow D=C/G$ is the quotient map. Then, for any integer $m$, 
\begin{equation}
\langle [1_G], \chi_G(\omega_C(T)^{\otimes m}) \rangle=\chi(\omega_D(\oT)^{\otimes m}) + \sum_{Q\in D\setminus \oT}\left\lfloor m\left(1-\frac{1}{|e_Q|}\right)\right\rfloor
\end{equation}
where for $Q\in D\setminus\oT$, $e_Q:=e_P=|\oG_P|$ is the ramification
index of $C\to D$ at any $P\in \pi^{-1}(Q)$. 
\end{prop}

\begin{proof} Keep the notation of Set-up~\ref{setup: curve}. 
Denote $\sE=(\omega_C(T))^{\otimes m}$. Since $I_i$ acts trivially on
$\sE|_{C_i}$, we have $\chi_{I_i}(\sE|_{C_i}) =
\chi(\sE|_{C_i})[1_{I_i}]$. By Corollary~\ref{cor: sm}, we have 
\[
\chi_G(\sE) = \sum_{1\leq i\leq n}\frac{|I_i|}{|G|} \Ind_{I_i}^G \chi(\sE|_{C_i})[1_{I_i}]+ \sum_{P\in C}\Gamma_G(\sE)_P
\]
Let $T_i=T\cap C_i$. Then for $P\in C_i\setminus T_i$, we have 
\[
[\sE|_P]=\left[\omega_C^{\otimes m}|_P\right] = \theta_P^{m_P}
\]
where $1\leq m_P\leq e_P$ satisfies $m_P\equiv m \mod e_P$, and hence
\[
\Gamma_G(\sE)_P=\Gamma_G(\omega_C^{\otimes m})_P = \frac{|I_i|}{|G|}\Ind_{G_P}^G
\left(\frac{e_P-1}{2}[k[\oG_P]]-\sum_{d=0}^{e_P-1}d\theta_P^{d+m_P}\right). 
\]
For any $P\in T_i$, we have (see \eqref{eq: theta_P})
\[
[\sE|_P]=\left[\omega_C(T)^{\otimes m}|_P\right] = [1_{G_P}]
\]
and hence
\[
\Gamma_G(\sE)_P = \frac{|I_i|}{|G|}\Ind_{G_P}^G
\left(\frac{e_P-1}{2}[k[\oG_P]]-\sum_{d=0}^{e_P-1}d\theta_P^d\right). 
\]
Thus, using the Frobenius reciprocity \eqref{eq: Frobenius2}, the virtual multiplicity of $[1_G]$ in $\chi_G(\sE)$ is
\begin{multline}\label{eq: 1G in E}
\langle [1_G], \chi_G(\sE) \rangle =  \sum_{1\leq i\leq n}\frac{|I_i|}{|G|}\chi(\sE|_{C_i})\langle[1_{I_i}], [1_{I_i}]\rangle+ \sum_{P\in C}\langle [1_G],\Gamma_G(\sE)_P\rangle \\
=\sum_{1\leq i\leq n} \frac{|I_i|}{|G|}\left( \chi(\sE|_{C_i}) +
  \sum_{P\in C_i} \frac{e_P-1}{2} - \sum_{P\in C_i\setminus T_i}
  (e_P-m_P) \right). 
\end{multline}
%% where $T_i=T\cap C_i$ for $1\leq i\leq n$.
By the Riemann--Roch theorem, we have
\begin{equation}\label{eq: RR}
\chi(\sE|_{C_i}) = \chi(\sO_{C_i}) + \deg\sE|_{C_i} 
\end{equation}
By the Riemann--Hurwitz formula, we have
\begin{equation}\label{eq: RH1}
  \chi(\sO_{C_i})  =\frac{|G_i|}{|I_i|}\chi(\sO_{\pi(C_i)}) - \frac{1}{2}\sum_{P\in C_i}(e_P-1).
\end{equation}
Note that $\sE|_{C_i} = \pi^*(\omega_{\pi(C_i)}(\oT_i)^{\otimes m})\otimes\sO_{C_i}(\sum_{P\in C_i\setminus T_i}(e_P-1)P)$, where $\oT_i=\pi(T_i)$, so 
\begin{equation}\label{eq: deg E|C_i}
 \deg\sE|_{C_i} = \frac{|G_i|}{|I_i|}\deg \omega_{\pi(C_i)}(\oT_i)^{\otimes m} + \sum_{P\in C_i\setminus T_i}(e_P-1).   
\end{equation}
By construction $e_P \mid m-m_P$. It is then easy to see that  
\begin{equation}\label{eq: e_P m_P}
  m(e_P-1)-(e_P-m_P)=e_P\lfloor m(1-1/e_P)\rfloor.     
\end{equation}
Assembling \eqref{eq: RR}--\eqref{eq: e_P m_P}, we get
\begin{multline*} 
  \chi(\sE|_{C_i})+\sum_{P\in C_i} \frac{e_P-1}{2} -
  \sum_{P\in C_i\setminus T_i} (e_P-m_P) = \\ 
  \frac{|G_i|}{|I_i|} \chi(\omega_{\pi(C_i)}(\oT_i)^{\otimes m}) +
   \sum_{P\in C_i\setminus T_i} e_P\lfloor m(1-1/e_P)\rfloor.  
 \end{multline*}
 Now notice that there are $|G|/|G_i|$ components of $C$ lying over 
 $\pi(C_i)$, and $|G_i/I_i|/e_P$ points of $C_i$ lying over $\pi(P)$. 
 We then get the desired equality by using \eqref{eq: 1G in E}.
 \end{proof}

 \subsection{The Chevalley--Weil formula for singular curves}\label{sec: CW singular}
Let the notation be as in Set-up~\ref{setup: curve}. 
\begin{defn}[Ramification modules on a singular $G$-curve]\label{defn: ramification module singular}
For each irreducible component $C_i$, we define
    \[
    \Gamma_G(\sE)_{C_i} := \Gamma_G(\nu^*\sE)_{\tC_i}
    \]
where $\tC_i$ is the irreducible component of $\tC$ lying over $C_i$.

For a point $P\in C$,  define
\[
\Gamma_G(\sE)_P =\sum_{\tP\in \nu^{-1}(P)}\Gamma_G(\nu^*\sE)_{\tP} + \frac{|G_P|}{|G|}\Ind_{G_P}^G\left(\frac{1}{|G_P|}\dim_k (\sE\otimes \sS)_P [k[G_P]] - [(\sE\otimes \sS)_P]\right)
\]
We note that $\Gamma_G(\sE)_P\neq 0$ for only finitely many points $P\in C$.
\end{defn}

\begin{thm}\label{thm: CW singular}
The following equality holds in $R_k(G)_\QQ$
   \begin{equation}\label{eq: CW singular}
\chi_G(\sE) = \frac{1}{|G|}\chi(\sE)[k[G]] + \sum_{Z\subseteq C} \Gamma_G(\sE)_Z
\end{equation}
where the sum is taken over all irreducible subvarieties $Z$ of $C$.
\end{thm}

\begin{proof}
By \eqref{eq: chi_G F locally free}, we have
    \begin{equation}\label{eq: chi E vs tE}
        \chi_G(\sE) = \chi_G(\nu_*\nu^*\sE) - [\sE\otimes \sS] =\chi_G(\nu^*\sE) - [\sE\otimes \sS]     
    \end{equation}
By Theorem~\ref{thm: sm}, we have
\begin{equation}\label{eq: CW singular1}
\chi_G(\nu^*\sE) = \frac{1}{|G|}\chi(\nu^*\sE)[k[G]] + \sum_{\tZ\subseteq \tC} \Gamma_G(\nu^*\sE)_{\tZ},  
\end{equation}
where $\tZ$ runs through the irreducible subvarieties of $\tC$. We have 
\begin{equation}\label{eq: CW singular2}
\chi(\nu^*\sE) =\chi(\sE) + \sum_{P\in C}\dim_k (\sE\otimes \sS)_P.
\end{equation}
Applying Corollary~\ref{lem: conn decomp}  to $\sE\otimes \sS$ over $\Supp \sS$, we have
\begin{equation}\label{eq: CW singular3}
[\sE\otimes \sS] = \sum_{P\in S}\frac{|G_P|}{|G|}\Ind_{G_P}^G [(\sE\otimes \sS)_P]
\end{equation}
Plugging \eqref{eq: CW singular1}--\eqref{eq: CW singular3} into \eqref{eq: chi E vs tE},  we obtain the desired expression~\eqref{eq: CW singular}.
\end{proof}

\begin{cor}\label{cor: E vs F}
  Let $C$
  %% =\bigcup_{1\leq i\leq n} C_i$
  be a (possibly disconnected) proper reduced curve over $k$, and
  $G\subseteq \Aut(C)$ a tame finite subgroup acting freely on
    a dense open subset of $C$. Let $\sE$
and $\sF$ be two locally free $G$-sheaves of rank $r$ such that
$[\sE\otimes \sS_P] = [\sF\otimes \sS_P]\in R_k(G_P)$ for all points $P\in C$. Then 
\[
\chi_G(\sE) - \chi_G(\sF) = \frac{1}{|G|}(\deg \sE - \deg\sF)[k[G]].
\]
\end{cor}

\begin{proof}
    Under the assumptions of the corollary, we have $\Gamma_G(\sE)_Z= \Gamma_G(\sF)_Z$ for any irreducible subvariety of $C$. By Theorem~\ref{thm: CW singular}, we obtain
    \[
    \chi_G(\sE) - \chi_G(\sF) =\frac{1}{|G|}(\chi(\sE) - \chi(\sF))[k[G]]= \frac{1}{|G|}(\deg \sE - \deg\sF)[k[G]]
    \]
\end{proof}

The following corollaries generalize \cite[Proposition~4.2, Corollary~4.3]{EL80}.

\begin{cor}
Let the assumptions be as in Corollary~\ref{cor: E vs F}. Let $\sM$ be a locally free sheaf of rank $s$ on $D$. Then 
\[
\chi_G(\sF\otimes\pi^*\sM) - s\chi_G(\sF) =  r(\deg\sM)[k[G]].
\]
In particular, $\chi_G(\pi^*\sM) -s\chi_G(\sO_C) =  (\deg\sM)[k[G]].$
\end{cor}

\begin{proof}
  Note that $[(\sF\otimes\pi^*\sM) \otimes \sS_P]=[(\sF\otimes \sO_{C, P}^{\oplus s})\otimes \sS_P] =
  [\sF^{\oplus s}\otimes\sS_P]\in R_k(G_P)$ for any $P\in C$, and $\deg(\sF\otimes\pi^*\sM) -\deg (\sF^{\oplus s})= r|G|\deg\sM$.
Now apply Corollary~\ref{cor: E vs F}.
\end{proof}

\subsection{An equivariant Riemann--Roch theorem}\label{sub:RR}
Let the notation be as in Set-up~\ref{setup: curve}. The classical 
Riemann--Roch theorem states that for any locally free sheaf $\sE$ of rank
$r$ on $C$, we have
\begin{equation}\label{eq: usual RR}
\chi(\sE)=r\chi(\sO_C)+\deg \sE.     
\end{equation}
In the equivariant setting, we would like to have a similar formula 
\[
\chi_G(\sE)=r\chi_G(\sO_C)+\deg_G \sE 
\]
for some $\deg_G\sE\in R_k(G)_\QQ$.
Bornes \cite[Th\'eor\`eme 4.10]{Bor03} established precisely such a formula  
when $C$ is smooth connected and $r=1$.

\begin{defn} \label{defn: deg_G E} We define the \emph{$G$-equivariant degree} of $\sE$ as
\begin{equation}\label{eq: deg_G E}
\deg_{G} \sE:=\frac{1}{|G|}(\deg\sE) [k[G]] + \sum_{Z\subseteq C}\Gamma_G(\sE)_Z - r\Gamma_G(\sO_C)_Z. 
\end{equation}
where $Z$ runs through all the irreducible subvarieties of $C$. In particular, when the $G$-action on $C$ is free, then $\deg_{G} \sE=\frac{1}{|G|}(\deg\sE) [k[G]]$.
\end{defn}

Combining the definition of $\deg_G \sE$ above, the usual Riemann--Roch formula \eqref{eq: usual RR}, and Theorem~\ref{thm: CW singular}, we immediately obtain

\begin{cor} \label{cor: Equiv-RR} 
$\chi_G(\sE)=r\chi_G(\sO_C)+\deg_G(\sE).$
\end{cor}

\begin{rmk} In \cite[\S 4.4]{Bor03}, Borne defined the equivariant
  degree for invertible $G$-sheaves over a projective smooth $G$-curve $C$, with a
  different approach (using Weil divisors).
  He computed $\chi_G(\sE)$ using his equivariant Riemann--Roch 
  theorem, while we define $\deg_G \sE$ using our formula
  for $\chi_G(\sE)$. 
\end{rmk}

Let $K_G^\circ(C)$ be the Grothendieck ring of locally free
$G$-sheaves of finite rank. Thanks to the additivity of $\chi_G(\sE)$ for $\sE$, the map $\sE\mapsto \deg_G\sE$ defines a group homomorphism from $K_G^\circ(C)$ to $R_k(G)$. However, this map does not respect tensor products on $K_G^\circ(C)$ and $R_k(G)$, as can be easily seen from the following examples.
  
\begin{ex}\label{ex: not additive}
Let $C$ be a connected smooth projective curve and $G\subseteq\Aut(C)$
a finite subgroup with nontrivial $G_P$ for some 
  $P\in C$.  For any $m\in \ZZ$, consider $\sE_m:=\sO_C(m\sum_{g\in G}
  g(P))$. Then we have $[\sO_C(mP)|_P]=\theta_P^{-m}$
  (see \eqref{eq: theta_P}). Denote by $e_P=|G_P|$, and let
  $1\leq m_P\leq e_P$ be such that $m_P\equiv m \mod e_P$. Then 
  %% see Definition~\ref{defn: ramification module sm}),  and hence, 
(using Equality~\eqref{eq:Gamma_G(E)_P} and Lemma~\ref{lem:MkG}) 
  \begin{align*}
\Gamma_G(\sE_m)_{P}-\Gamma_G(\sO_C)_{P}  
&=\frac{1}{|G|}\Ind_{G_P}^G(\Psi_P\otimes ([\sO_C(mP)|_P]-[1_{G_P}] ))\\
& =
\frac{1}{|G|}\Ind_{G_P}^G\left(\sum_{d=0}^{e_P-1}
d\theta_P^d -d\theta_P^{d-m_P}\right)\\
 & = \frac{1}{|G|}\Ind_{G_P}^G
\left(e_P\sum_{d=e_P-m_P}^{e_P-1}\theta_P^d -m_P\sum_{d=0}^{e_P-1}\theta_P^d\right)\\
              &= -\frac{m_P}{|G|}[k[G]] + \frac{e_P}{|G|}\Ind_{G_P}^G
                \left(\sum_{d=e_P-m_P}^{e_P-1}\theta_P^d\right) 
\end{align*}
because $\sum_{d=0}^{e_P-1}\theta_P^d=k[G_P]$. Therefore 
\begin{align*}
\deg_G(\sE_m) &=\frac{m}{e_P}[k[G]]+\sum_{P'\in G\cdot P}\Gamma_G(\sE_m)_{P'}-\Gamma_G(\sO_C)_{P'} \\
%              &=\frac{m}{e_P}[k[G]]-\frac{m_P}{e_P}[k[G]]+\Ind_{G_P}^G
%                \left(\sum_{d=e_P-m_P}^{e_P-1}\theta_P^d\right) \\ 
& =
\frac{m-m_P}{e_P}[k[G]]+ \Ind_{G_P}^G
\left(\sum_{d=e_P-m_P}^{e_P-1}\theta_P^d\right).
\end{align*}
By looking at the multiplicities of $[1_G]$ and $[\theta_P]$ on both sides,  it is evident that  $\deg_G(\sE_m) \ne m\deg_G(\sE_1)$ for  $m>1$.
\end{ex}

\begin{ex}  Suppose $C=\bigcup_{1\leq i\leq n} C_i$ is smooth, and $T\subset C$ is a finite $G$-subset. Let $m$ be an integer, and write $m=r|\oG_P|+m_P$ with $r\in \ZZ$ and $1\leq m_P\leq |\oG_P|$. Then we have
  \begin{multline*}
 \deg_G \omega_C(T)^{\otimes m}=\frac{m\deg\omega_C(T)}{|G|}[k[G]] \\ 
  +\frac{1}{|G|} \sum_{P\in C\setminus T} \Ind_{G_P}^G\left(m_P[k[\oG_P]]-|\oG_P|\sum_{d=0}^{m_P-1}\theta_P^d\right) \\
  +\sum_{i=1}^n \frac{m|I_i|}{|G|}\Ind_{I_i}^G(\deg\omega_{C_i})\left([1_{I_i}] -\frac{1}{|I_i|}[k[I_i]]\right).     
  \end{multline*}
  In particular, if $m$ is divisible by $|G|$ and the inertia group
  $I_i$ is trivial for each component $C_i$, then $ \deg_G
  \omega_C(T)^{\otimes
    m}=\frac{m}{|G|}(\deg\omega_C(T))[k[G]]$; this
    is also a consequence of Corollary~\ref{cor: E vs F}, as the action of $G_P$ is
    then trivial on $\omega_{C}(T)^{\otimes m}|_P$ for any $P\in C$.

  \begin{proof} For $P\in T$, we have $[\omega_C(T)^{\otimes m}|_P]=[1_{G_P}]$, so 
  \[
  \Gamma_G(\omega_C(T)^{\otimes m})_P-\Gamma_G(\sO_C)_P=0.
  \]
  For $P\in C\setminus T$, we have
  $[\omega_C(T)^{\otimes m}|_P]=\theta_P^m$.  Similar computations as in Example~\ref{ex: not additive} give
  \begin{equation*}\label{eq: Gamma Omega vs O}
 \begin{split}
       \Gamma(\omega_C(T)^{\otimes m})_P - \Gamma(\sO_C)_P
     & = \frac{1}{|G|}\Ind_{G_P}^G\left(\sum_{d=0}^{ |\oG_P|-1} -d\theta_P^{d+m_P}+d\theta_P^{d}\right) \\
 %%  &=\frac{1}{|G|} \Ind_{G_P}^G\left( m_P\sum_{d=m_P}^{ |\oG_P|-1} \theta_P^{d} - (|\oG_P|-m_P)\sum_{d=0}^{m_P-1} \theta_P^{d}\right) \\
 &   =\frac{1}{|G|} \Ind_{G_P}^G \left(m_P[k[\oG_P]] - |\oG_P|\sum_{d=0}^{m_P-1} \theta_P^{d}\right). 
\end{split}
  \end{equation*}
For an irreducible component $C_i$ of $C$, we have  $\chi_{I_i}(\sE|_{C_i})=\chi(\sE|_{C_i})[1_{I_i}]$. By Definition~\ref{defn: ramification module sm}, setting $T_i=T\cap C_i$,
  \begin{multline*}
    \Gamma_G(\omega_C(T)^{\otimes m})_{C_i} -\Gamma_G(\sO_C)_{C_i} 
    \\ 
=    \frac{|I_i|}{|G|}\Ind_{I_i}^G(\chi(\omega_{C_i}(T_i)^{\otimes m})-\chi(\sO_{C_i}))([1_{I_i}] -\frac{1}{|I_i|}
 [k[I_i]])\\
  = \frac{m|I_i|}{|G|}\Ind_{I_i}^G (\deg\omega_{C_i}(T_i))\left([1_{I_i}] -\frac{1}{|I_i|}[k[I_i]]\right).      
  \end{multline*}
 Summing up, we obtain the desired formula for $\deg_G(\omega_C(T)^{\otimes m})$ by the definition given in \eqref{eq: deg_G E}.
  \end{proof}
\end{ex}

\section{The case of nodal curves}\label{sec: nodal}
In this section, we keep the notation of Set-ups~\ref{setup: curve}. Additionally, we assume that $C$ %% =\bigcup_{1\leq i\leq n} C_i$
has at worst nodes as singularities.

\begin{setup}\label{setup: nodal} Let $P\in S$. Because of the tameness of the $G$-action, $|G_P|$ is nonzero in $k$, and we may find a $G_P$-equivariant isomorphism of $k$-algebras between the formal completion $\widehat\sO_{C,P}$ and $k[[x,y]]/(x,y)$, where $G_P$ acts linearly on $k[[x,y]]/(x,y)$; see \cite{Car57}.

Let $\{ \tP_1, \tP_2\}=\nu^{-1}(P)\subset \tC$  denote the preimages. 
  The canonical $k$-linear map
\[T_{\tC, \tP_1}\oplus T_{\tC, \tP_2}\longrightarrow T_{C, P}\] 
of Zariski tangent spaces 
is an isomorphism.   The subspaces $T_{\tC, \tP_1}, T_{\tC, \tP_2}$  of
$T_{C, P}$ are referred to as the
\emph{local branches} of $C$ at $P$. 

The group $G$ acts on $\tC$ and we
have a canonical exact sequence of $G$-equivariant sheaves 
\begin{equation}
  \label{eq:OC-OtC}
  0 \to \sO_C \to \nu_* \sO_{\tC} \to \sS \to 0,   
\end{equation}
where $\sS$ is supported in the  singular locus $S$ of $C$, with
$\sS_P\simeq k$ for all $P\in S$ (because $P$ is a node). 
We define two subsets of $S$: 
\begin{equation}\label{eq: subsets S}
  \begin{split}
& S_1 = \{P\in S \mid \text{$G_P$ does not exchange the local branches
  at $P\in C$} \}\\
& S_2 = \{P\in S \mid \text{$G_P$ exchanges the local branches at $P\in C$} \}
\end{split}   
\end{equation}
%\\ & S^{\mathrm{irr}} = \{ P\in S \mid C  \text{ is irreducible  at } P \}. \  \ 
%% S^{\mathrm{irr}}_1=S^{\mathrm{irr}}\cap S_1.  
%% 
Then $S$ is the disjoint union of $S_1$ and $S_2$. For any $P\in C$,
the condition $P\in S_1$ is equivalent to the condition that $\pi(P)$ is singular in
$D:=C/G$, where $\pi\colon C\rightarrow D$ denotes the quotient map.

For $P\in S_2$, let $G_{\tP}$ be the subgroup of index $2$ in $G_P$ fixing
  each local branch at $P$. It is also the stabilizer at any point
  $\tP\in \nu^{-1}(P)$. We denote by
  \begin{equation}\label{eq: chi_P}
    \chi_P=[k[G_P/G_{\tP}]]-[1_{G_P}] \in R_k(G_P).
  \end{equation}
This is the class of the $G_P$-representation of degree $1$ given by the character  
$G_P\to G_P/G_{\tP}\simeq \{ \pm 1 \}\hookrightarrow k^*$
\begin{rmk}
By the tameness of the $G$-action, if $S_2\neq \emptyset$, then $\chara(k)\neq 2$.
\end{rmk}
 
\begin{lem}\label{lem: S_P} We have in $R_k(G_P)$:
  \[ 
 [\omega_C|_P] = [\sS|_P]=
  \begin{cases}
    [1_{G_P}] &  \text{if } P\in S_1, \\
   \chi_P   & \text{if } P \in S_2.
  \end{cases}
  \]
  %% where $\chi_P$ is defined as in \eqref{eq: chi_P}.
  %If $\chara\, k=2$, then $[\omega_C|_P]=[\sS|_P] = [1_{G_P}]$ for any $P\in S$.
\end{lem}

\begin{proof}
Let $\{ \tP_1, \tP_2 \}=\nu^{-1}(P)$. We have the exact sequence
\[
0 \to k(P) \to k(\tP_1) \oplus k(\tP_2) \to \sS|_P \to 0. 
\] 
If $P\in S_1$, then $[k(\tP_j)]=[k(P)]$ as $k[G_P]$-modules, thus 
$[\sS|_P]=[k(P)]=[1_{G_P}]$.
If $P\in S_2$. Then 
$[k(\tP_1) \oplus k(\tP_2)]=k[G_P/G_{\tP}]$, thus 
$[\sS|_P]=k[G_P/G_{\tP}] -[1_{G_P}] $.

% If $\chara\, k=2$, then for any $2$-regular $g\in G_P$ (that is, $g$ has odd order), the Brauer characters of both $\sS|_P$ and $1_{G_P}$ take the same value $1$ at $g$. So $[\sS|_P] = [1_{G_P}]\in R_k(G_P)$ by \cite[Corolloary~17.10]{CR90}.

The proof for $[\omega_C|_P]$ is similar and hence is omitted.
\end{proof}   
\end{setup}

\subsection{Computation of \texorpdfstring{$\Gamma_G(\sE)_P$}{} at nodes with smoothable  
  actions} \label{sub:compute_smoothable}

Keep the notation as in Set-ups~\ref{setup: curve} and \ref{setup: nodal}. Recall that $G$ acts tamely on $C$. In this subsection, we assume that the action of $G$ is free on an open dense subset of $C$. For a singular point $P$ of $C$, let $\{ \tP_1, \tP_2\}=\nu^{-1}(P)\subset \tC$, and $x_i\in T_{\tP_i}\tC$ a nonzero element, $i=1,2$. Since the action of $G$ is free on an open dense subset of $C$, we have $G_{\tP_1}=G_{\tP_2}$ and they are cyclic.  
\begin{defn} \label{defn: smoothable} 
The action of $G_P$ is called \emph{smoothable} at a singular point $P\in C$ as above if for any $g\in G_P$, the pair of cotangent vectors $(g^*x_1, g^*x_2)$ is either $(\xi x_1, \xi^{-1} x_2) $ or $(\xi x_2, \xi^{-1} x_1)$ for some $\xi\in k$, depending $g$ preserves the local branches of $C$ at $P$ or not. We say that the action of $G$ on $C$ is 
\emph{smoothable} if the above condition holds for all $P\in S$. 
\end{defn}

\begin{rmk}
By \cite[Section~2.1]{Liu12}, the $G$-curve $C$ is smoothable if and only if there exists a $G$-surface $\sC$ and a flat $G$-equivariant
  morphism $f\colon \sC \to T$ to an integral smooth curve $T$ (with
  trivial action of $G$) such that the following holds: 
\begin{itemize}
\item the generic fiber of $\sC\to T$ is smooth;
\item there exist $t_0\in T$ and a $G$-equivariant isomorphism between the
   curves $\sC_{t_0}$ and $C$. 
\end{itemize}    
\end{rmk}

% \begin{lem}
% Suppose that the action of $G$ on $C$ is smoothable.
% Then two irreducible components of $C$ lying on the same connected component have the same inertia groups.
% \end{lem}
% \begin{proof}
% Suppose that $C_{i_1}, C_{i_2}$ are the irreducible components of $C$ passing the same point $P$.
% The smoothability condition implies that $I_{i_1}=I_{i_2}$. By induction, the conclusion of the lemma holds.
% \end{proof}

We expect that its ramification module of a smoothable $G$-curve resembles that of a smooth $G$-curve.
\begin{prop}\label{prop:GammaP_ad}
Assume that the action of $G$ on $C$ is smoothable. Let $P\in S$, and pick a point $\tP\in \nu^{-1}(P)=\{\tP_1,\tP_2\}$.
\begin{enumerate}
  \item We have 
\[
\sum_{i=1,2}\Gamma_{G}(\nu^*\sE)_{\tP_i}
=  \frac{|G_{\tP}|}{|G|}\Ind_{G_{\tP}}^G[\sE|_P]
-\frac{r}{|G|} [k[G]].
\]
\item If  $P\in S_1$, then  $\Gamma_G(\sE)_P=0$. 
\item If $P\in S_2$ then
    \[
\Gamma_G(\sE)_P= \frac{|G_{P}|}{2|G|}
\Ind_{G_P}^G ([\sE|_P]\otimes([1_{G_P}]-\chi_P)). 
\]
% In particular, $\Gamma_G(\sE)_P=0$ if $\chara(k)=2$. 
  \end{enumerate}
\end{prop} 

\begin{proof} (1) The smoothability condition implies that $G_{\tP_1}=G_{\tP_2}=G_{\tP}$, and
  $\theta_{\tP_1}=\theta_{\tP_2}^{-1}=\theta_{\tP_2}^{|G_{\tP}|-1}$, so 
\[
\sum_{i=1,2}\sum_{d=0}^{|G_{\tP}|-1} d\theta_{\tP_i}^d=\sum_{d=0}^{|G_{\tP}|-1} d\theta_{\tP}^d + d\theta_{\tP}^{|G_{\tP}|-d} = |G_{\tP}|\sum_{d=1}^{|G_{\tP}|-1}\theta_{\tP}^d = |G_{\tP}|([k[G_{\tP}]]-[1_{G_{\tP}}])
\]
Note that $[(\nu^*\sE)|_{\tP_i}] = [\sE|_P]\in R_k(G_{\tP_i})$, where the latter $k[G_P]$-module becomes a $k[G_{\tP_i}]$-module via the inclusion $G_{\tP_i}\hookrightarrow G_P$.
Therefore, by Definition~\ref{defn: ramification module sm} and Lemma~\ref{lem:MkG},
  \begin{align*}
\sum_{i=1,2}\Gamma_{G}(\nu^*\sE)_{\tP_i} &= -\frac{1}{|G|}\sum_{i=1,2}\Ind_{G_{\tP}}^G\sum_{d=0}^{|G_{\tP}|-1} d\theta_{\tP_i}^d\otimes [\sE|_P]\otimes\left([1_{G_{\tP}}] - \frac{1}{|G_{\tP}|}[k[G_{\tP}]]\right)\\
&=\frac{1}{|G|}\Ind_{G_{\tP}}^G|G_{\tP}|([1_{G_{\tP}}]-[k[G_{\tP}]])\otimes [\sE|_P]\otimes\left([1_{G_{\tP}}] - \frac{1}{|G_{\tP}|}[k[G_{\tP}]]\right) \\
&= \frac{1}{|G|}\Ind_{G_{\tP}}^G\left(|G_{\tP}|[1_{G_{\tP}}] - [k[G_{\tP}]]\right)\otimes [\sE|_P]
\end{align*}
Since $[k[G_{\tP}]] \otimes [\sE|_P]= r[k[G_{\tP}]]$ by Lemma~\ref{lem:MkG}, and $\Ind_{G_{\tP}}^G[k[G_{\tP}]] = [k[G]]$, the assertion of (1) follows.

\medskip

(2)-(3) By Definition~\ref{defn: ramification module singular}, we have
\begin{align*}
\Gamma_G(\sE)_P & =\sum_{\tP\in \nu^{-1}(P)}\Gamma_G(\nu^*\sE)_{\tP} +
\frac{|G_P|}{|G|}\Ind_{G_P}^G\left(r[k[G_P]]-[(\sE\otimes \sS)|_P] \right) \\
&=\frac{|G_{\tP}|}{|G|}\Ind_{G_{\tP}}^G[\sE|_P]
-\frac{r}{|G|} [k[G]]+
\frac{|G_P|}{|G|}\Ind_{G_P}^G\left(\frac{r}{|G_P|}[k[G_P]]-[(\sE\otimes \sS)|_P]\right)                               \\
&=\frac{|G_{\tP}|}{|G|}\Ind_{G_{\tP}}^G[\sE|_P] -\frac{|G_{P}|}{|G|}\Ind_{G_P}^G[(\sE\otimes \sS)|_P]
\end{align*}
If $P\in S_1$, then $G_{\tP}=G_P$ and $[\sS|_P]=[1_{G_P}]$
(Lemma~\ref{lem: S_P}). This implies that $\Gamma_G(\sE)_P=0$.

Suppose now $P\in S_2$. Then $|G_{\tP}|=|G_P|/2$, $[\nu^*\sE|_{\tP}]=\Res^{G_P}_{G_{\tP}} [\sE|_P]$, and 
$[\sS|_P]=\chi_P$  (Lemma~\ref{lem: S_P}).
Applying Lemma~\ref{lem:ind_res}(3) to 
$\Ind_{G_{\tP}}^{G_P}[\nu^*\sE|_{\tP}]$ we get
\[
\Ind_{G_{\tP}}^{G_P} [\nu^*\sE|_{\tP}]=[k[G_P/G_{\tP}]] \otimes [\sE|_P],
\]
hence 
\[
\Gamma_G(\sE)_P= \frac{|G_{P}|}{2|G|}
\Ind_{G_P}^G([\sE|_P]\otimes ([k[G_P/G_{\tP}]]-2\chi_P)). 
\]
By definition, $\chi_P=[k[G_P/G_{\tP}]]-[1_{G_P}]$. This achieves the proof. 
\end{proof}

\subsection{Lower bound on \texorpdfstring{$\chi_G(\omega_C\otimes \sA)$}{} for ample
  sheaves \texorpdfstring{$\sA$}{}}
%%   when $\sA$ enjoys certain   positivity property}

Keep the notation in Set-ups~\ref{setup: curve} and \ref{setup:
  nodal}. So $C$ is a nodal curve with $G$ acting tamely on  $C$.
In this subsection, we write $\sE=\omega_C\otimes\sA$ and try to bound $\chi_G(\sE)$ from below in terms of the quotient curve $D$ and $\deg\sA|_{C_i}$ for each component $C_i$. Under the conditions that the inertia group $I_i$ for of each irreducible component $C_i$ is trivial, and that $\sA$ is sufficiently ample, one can find some positive rational number $a$ such that $\chi_G(\sE)\geq a\cdot [k[G]]$ in $R_k(G)_\QQ$; see Corollary~\ref{cor: H0 omega tC tensor A vs k[G]}.

\begin{lem} \label{lem:chi_GsL} 
Let $\sE$ be a locally free $G$-sheaf on $C$ and set $\sA = \sE\otimes \omega_C^{-1}$. Then we have
\[ \chi_G(\sE)=\chi_G(\omega_{\tC}\otimes\nu^*\sA)+[H^0(S, \sE|_S)]. 
\]
 If $\sA$ is ample, then 
\[ \chi_G(\sE)=[H^0(C, \sE)],\,\, \chi_G(\omega_{\tC}\otimes\nu^*\sA)
  = [H^0(\tC, \omega_{\tC}\otimes \nu^*\sA)] \]
and $\chi_G(\nu^*\sE)=[H^0(\tC, \nu^*\sE)]$. 
\end{lem}

\begin{proof} Consider the exact sequence (\cite[Lemma 10.3.12]{Liu06})
\begin{equation}\label{eq: omegaC-omegatC}
    0\rightarrow \nu_*\omega_{\tC} \rightarrow \omega_C \rightarrow \omega_C|_S \rightarrow 0.
\end{equation}
Tensoring \eqref{eq: omegaC-omegatC} with $\sA$,
we obtain the $G$-equivariant exact sequence 
\begin{equation}\label{eq: omegaC-omegatC tensor E}
    0\rightarrow (\nu_*\omega_{\tC})\otimes \sA \rightarrow \sE \rightarrow  \sE|_S \rightarrow 0.
  \end{equation}
Note that $\chi_G((\nu_*\omega_{\tC})\otimes \sA) =\chi_G(\omega_{\tC}\otimes \nu^*\sA)$. \eqref{eq: omegaC-omegatC tensor E} implies that the first desired equality in $R_k(G)$ by taking the long
cohomology sequence of \eqref{eq: omegaC-omegatC tensor E}.

Suppose that $\sA$ is ample. Since $\nu$ is a finite morphism, $\nu^*\sE$ is also ample. By \cite{Har71}, we have the vanishing of $H^0(C, \sA^{\vee})$ and $H^0(\tC, (\nu^*\sA)^\vee)$.
By Serre's duality, we have $H^1(C, \sE)\simeq H^0(C, \sA^{\vee})$, and $H^1(C, (\nu_*\omega_{\tC})\otimes \sA)\simeq H^1(\tC,  
\omega_{\tC}\otimes \nu^*\sA)=H^0(\tC, (\nu^*\sA)^{\vee})$, and they vanish.
\end{proof}

\begin{lem}\label{lem: chi_G omega + A} 
  Suppose that $C$ is smooth,  connected and that $G\subseteq \Aut(C)$ is a tame finite subgroup. Let $\sE$ be a locally free $G$-sheaf of rank $r$ on $C$, and $\sA:=\omega_C^{-1}\otimes \sE$. 
  Then  
    \[
\chi_G(\sE)\geq \left(r\chi(\omega_D)+\frac{1}{|G|}\deg \sA\right)[k[G]].  
    \]
% In particular, if $\deg\sA>-r|G|\chi(\omega_D)$, then $\chi_G(\sE)$ contains a positive rational multiple of $[k[G]]$.
  \end{lem}

\begin{proof} Let $R$ be the ramification divisor of the quotient map $\pi\colon C\to
  D$.  By the Riemann--Hurwitz formula and the tameness hypothesis on the $G$-action on $C$, we have 
  \[ \chi(\omega_C)=|G|\chi(\omega_D)+\frac{1}{2}\deg R.\] By the Riemann--Roch theorem, we have
  \[
  \chi(\sE) = r\chi(\sO_C) + \deg\sE =  r\chi(\sO_C) + r\deg\omega_C + \deg\sA = r\chi(\omega_C)+\deg\sA
  \]
By Theorem~\ref{thm: sm}, 
  \begin{align*}
  \chi_G(\sE) & = \frac{1}{|G|}\chi(\sE)[k[G]]+\sum_{P\in C}\Gamma(\sE)_P \\
    & = \left(r\chi(\omega_D)+\frac{1}{|G|}\deg \sA\right)[k[G]] + \frac{r}{2|G|}(\deg R) [k[G]] 
      +\sum_{P\in C}\Gamma(\sE)_P.  
  \end{align*} 
  As $\deg R = \sum_{P\in C} (|G_P|-1)$, and $\displaystyle r[k[G_P]]=[k[G_P]]\otimes [\sE|_P]= \sum_{d=0}^{|G_P|-1}
\theta_P^d\otimes [\sE|_P]$, we have by Definition~\ref{defn: ramification module sm}
\begin{align*}
   &\frac{r}{2|G|}(\deg R) [k[G]]   +\sum_{P\in C}\Gamma(\sE)_P  \\
   & =\frac{1}{|G|}\sum_{P\in C} 
\Ind_{G_P}^G\left(r(|G_P|-1)[G_P]]
- \left( \sum_{d=0}^{|G_P|-1} d\theta_P^d\otimes [\sE|_P]\right)\right) \\
&= \frac{1}{|G|} \sum_{d=0}^{|G_P|-1}(|G_P|-1-d)\theta_P^d \otimes [\sE|_P]\\
&  \geq 0
\end{align*}
Thus $\chi_G(\sE)\geq \left(r\chi(\omega_D)+\frac{1}{|G|}\deg
  \sA\right)[k[G]]$, as desired. 
\end{proof}

\begin{cor}\label{cor: g(C/G)>1}
Let $C$ be a smooth connected projective curve over $k$, and
$G\subseteq \Aut(C)$ a finite tame subgroup. If $g(C/G)\geq 2$, then
\[ [H^0(C, \omega_C)]\geq [k[G]].\]
In particular, if
  $|G|$ is prime to the characteristic exponent of $k$,
then each $V\in \Irr_k(G)$ has positive
multiplicity in  the $G$-module $H^0(C,\omega_C)$.
\end{cor}

\begin{proof}
    Take $\sE=\omega_C$ in Lemma~\ref{lem: chi_G omega + A}, and note that 
    \[
    [H^0(C,\omega_C)] = \chi_G(\omega_C) + [H^1(C, \omega_C)] = \chi_G(\omega_C) + [1_G]>\chi_G(\omega_C).
    \]
\end{proof}
The conclusion of Corollary~\ref{cor: g(C/G)>1} does not hold any more in the singular case. We provide in the following two examples for this phenomenon.

\begin{ex}
Let $C=C_1\cup C_2$ be a stable curve consisting of two smooth
components such that $C_1\cap C_2=\{P\}$, and that $\Aut(C_2)$
  contains a cyclic subgroup $G$ of order $n>g(C_2)+1$. Let $G$ act on
  $C$ as follows: 
    \begin{enumerate}
        \item $G$ preserves the components $C_1$ and $C_2$;
        \item $G$ acts trivially on $C_1$ and as subgroup of
          $\Aut(C_2)$  on $C_2$. 
    \end{enumerate} 
    Then $H^0(C, \omega_C) = H^0\left(C_1, \omega_{C_1}\right) \oplus H^0\left(C_2, \omega_{C_2}\right) $ and $G$ acts trivially on the first summand $H^0\left(C_1, \omega_{C_1}\right)$. Now there are $n-1$ nontrivial irreducible characters of $G$, while $\dim H^0\left(C_2, \omega_{C_2}\right)=g(C_2)<n-1$ by assumption, there must be some nontrivial irreducible character of $G$ not present in $H^0\left(C_2, \omega_{C_2}\right)$ and hence not in $H^0(C, \omega_C)$ either. On the other hand, the arithmetic genus of the quotient curve is
    \[
    p_a(C/G) =g(C_1/G) + g(C_2/G) \geq g(C_1/G) = g(C_1)
    \]
    which can be chosen arbitrarily large.
\end{ex}

In the following example the curve $C$ is even irreducible.

\begin{ex}
Let $p\geq 5$ be a prime number different from $\mathrm{char}(k)$, and $G=\langle\sigma\rangle\cong \ZZ/p\ZZ$. By the Riemann existence theorem, we may construct a $G$-cover $\tilde\pi\colon\tC\rightarrow \PP^1$ with $4$ branch points $Q_j, 1\leq j\leq 4$, over which the monodromies are given by the sequence 
    \[
   \left(\sigma,\, \sigma,\, \sigma,\,\sigma^{-3}\right).
    \]
Let $\tP_j$ be the inverse image of $Q_j$ for $1\leq j\leq 4$, which are fixed by $G$. Recall that 
\[
\theta_{\tP_j}=\left[T_{\tP_j}^* \tC\right]\in R_k(G), \quad 1\leq j\leq 4.
\]
For $r\in \ZZ$, define $\chi_r:=\theta_{\tP_1}^r$. Then the $\{\chi_r\}_{1\leq r\leq p}$
are exactly  the classes of the irreducible $G$-modules, with $\chi_p = [1_G]$. We have
\[
\theta_{\tP_1} = \theta_{\tP_2} = \theta_{\tP_3} =\theta_{\tP_4}^{-3}=\chi_{1} 
\]
One computes easily by the Riemann--Hurwitz formula that $g(\tC) =
p-1$. By Theorem~\ref{thm: sm}, we have 
\[
[H^0(\tC, \omega_{\tC})]=\chi_G(\omega_{\tC})+[1_G] = [1_G] + \frac{p-2}{p}[k[G]]+\frac{1}{p}\sum_{j=1}^4 \left(\frac{p-1}{2}[k[G]]-\sum_{d=0}^{p-1}d\theta_{\tP_j}^{d+1}\right)
\]
We have
  $\langle \chi_r, \theta_{\tP_4}^{d+1}\rangle=1$ if $d+1\equiv -3r
  \mod p$ and is zero otherwise.
It follows that
    \begin{equation}\label{eq: eigenspaces}
        \mu_{\chi_r}H^0(\tC, \omega_{\tC}) =\langle \chi_r, [H^0(\tC, \omega_{\tC})]\rangle = 
    \begin{cases}
        2, & \text{if $1\leq r\leq \lfloor\frac{p-1}{3}\rfloor$},  \\
        1, & \text{if $\lceil\frac{p}{3} \rceil\leq  r \leq \lfloor \frac{2p-1}{3} \rfloor$}, \\
        0, & \text{if $\lceil \frac{2p}{3}\rceil \leq r \leq p$}. \\
    \end{cases} 
    \end{equation}
    Let $\nu\colon \tC\rightarrow C$ be the gluing of the pairs of points $\{\tP_1, \tP_2\}$ and $\{\tP_3, \tP_4\}$, so that $C$ is a stable curve with two nodes $P_1=\nu(\tP_1)$ and $P_2=\nu(\tP_3)$. Since the points $\tP_j$ are fixed by the $G$, the action of $G$ descends to $C$, preserving the local branches at $P_1$ and $P_2$.  By \eqref{eq: omegaC-omegatC}, we have 
    \[
    [H^0(C, \omega_C)] = \left[H^0(\tC, \omega_{\tC})\right] + [\omega_C|_S]= \left[H^0(\tC, \omega_{\tC})\right] + 2[1_G].
    \] 
Thus for $\lceil \frac{2p}{3}\rceil \leq r \leq p-1$, we have
    \[
 \mu_{\chi_r}H^0(C, \omega_C) =    \mu_{\chi_r}H^0(\tC, \omega_{\tC}) = 0.
    \]
 On the other hand, the quotient curve $D=C/G$ is a rational curve
 with two nodes, and hence $p_a(D)=2$. However, $H^0(C, \omega_C)$ does
 not contain any copy of $\chi_r$. 
\end{ex}

\begin{cor}\label{cor: H0 omega tC tensor A vs k[G]}
Let $C = \bigcup_{1\leq i\leq n} C_i$ be a proper nodal curve over $k$, and $G\subseteq\Aut(C)$ a
  tame finite subgroup such that the inertia group $I_i$ for each component $C_i$ is trivial.
  %% Let $\sA$ be an ample $G$-equivariant invertible sheaf on $C$.
  Let $\sE$ be a locally free $G$-sheaf of rank $r$ on $C$, and denote $\sA=\sE\otimes \omega_C^{-1}$. Suppose that $\sA$ is ample. Then the following holds.
  \begin{enumerate}
  \item For any $P\in S$, we have 
    $\chi_G(\sE)\geq \Ind_{G_P}^G [\sE|_P]$.  
In particular if $G$ acts freely on the orbit of a singular point, then
$\chi_G(\sE)\ge r[k[G]]$. 
  \item For all $i\leq n$, we have 
    \[ \chi_G(\sE)\geq \left(r\chi(\omega_{\tD_i})+\frac{1}{|G_i|}\deg
        (\sA|_{C_i})\right)[k[G]]. 
    \]
    where $G_i$ is the stabilizer of $C_i$ (Set-ups~\ref{setup: curve}), and $\tD_i=\tC_i/G_i$ is the quotient curve.
    In particular, if  $\deg \sA|_{C_i}> -r|G_i|\chi(\omega_{\tD_i})$
    (automatic if $g(\tD_i)> 0$), 
    then any irreducible representation 
    of $G_i$ has positive virtual multiplicity  in $\chi_G(\sE)$.     
\item If $\deg \sA|_{C_i}\geq  |G_i|(-r\chi(\omega_{\tD_i})+1)$
  (automatic if $g(\tD_i)\geq 2$) for some $i\leq n$,  then $\chi_G(\sE)\geq [k[G]]$.
\end{enumerate}
\end{cor}

\begin{proof}  (1) 
 Let $P\in S$. Then $\bigoplus_{P'\in G.P} \sE|_{P'}$ is a direct
summand  of $\sE|_S$. By Lemma~\ref{lem:chi_GsL}, 
\[
\chi_G(\sE)=[H^0(\tC, \omega_{\tC}\otimes \nu^*\sA)] + [\sE|_S] \geq [\sE|_S]\geq \sum_{P'\in G.P} [\sE|_{P'}]=\Ind_{G_P}^G[\sE|_P].   
 \] 
 If $G_P=\{1\}$, then the right-hand side is equal to $r[k[G]]$.  

 (2) By Lemma~\ref{lem:chi_GsL},
 $\chi_G(\sE)\geq \chi_G(\omega_{\tC}\otimes \nu^*\sA)$.
 As $\nu^*\sA$ is ample on $\tC$ with $\deg \nu^*\sA|_{\tC_i}=\deg
 \sA|_{C_i}$,  it is enough to work over $\tC$, hence we can
 suppose $C$ is smooth. 
 Fix $i\le n$. Then
 $\bigoplus_{C_{j} \subseteq  GC_i} H^0(C_j, \sE|_{C_j})$
 is a direct summand of $H^0(C, \sE)$,  hence 
 \[
\chi_G(\sE)=[H^0(C, \sE)]\geq \left[\bigoplus_{C_j \subseteq GC_i} H^0(C_j, \sE|_{C_j})\right]
   =\Ind_{G_i}^G [H^0(C_i, \sE|_{C_i})].  
 \]
It suffices to apply Lemma~\ref{lem: chi_G omega + A} to $(C_i, \sE|_{C_i})$
with the tame subgroup $G_i\subseteq \Aut(C_i)$.

(3) is an immediate consequence of (2). 
\end{proof}

\subsection{The \texorpdfstring{$G$}{}-invariant part of \texorpdfstring{$H^0(C, \omega_C(T)^{\otimes m})$}{}}
Let the notation be as in Set-ups~\ref{setup: curve} and \ref{setup: nodal}. 

\begin{thm}\label{thm: pluricanonical G-inv}
Assume that $G$ acts faithfully on $C$. Let $T$ be a finite $G$-set contained in the smooth locus of $C$. Then for any integer $m$, we have
\[
\langle [1_G], \chi_G(\omega_C(T)^{\otimes m})\rangle = \chi(\omega_D(\oT)^{\otimes m}) +(m - \epsilon_m) \#\oS_2+\sum_{Q\in D\setminus (\oS\cup\oT)}\left\lfloor m\left(1-\frac{1}{|e_{Q}|}\right)\right\rfloor   
\]
where $\oS,\, \oS_2, \oT\subset D$ are the images of $S$, $S_2$ and $T$ respectively, for $Q\in D\setminus(\oS\cup \oT)$, $e_Q:=|G_P|$ for $P\in \pi^{-1}(Q)$, and 
\begin{equation}\label{eq: epsilon_m}
  \epsilon_m =
\begin{cases}
1 & \text{if $m$ is odd} \\
0 & \text{if $m$ is even}
\end{cases}  
\end{equation}
\end{thm}
\begin{proof}
Tensoring the short exact sequence \eqref{eq:OC-OtC} with $\omega_C(T)^{\otimes m}$, we obtain another $G$-equivariant short exact sequence
\[
0\rightarrow \omega_C(T)^{\otimes m} \rightarrow \nu_*\nu^*\omega_C(T)^{\otimes m}\rightarrow \sS\otimes  \omega_C(T)^{\otimes m}\rightarrow 0
\]
Set $\tS=\nu^{-1}S$ and $\tT = \nu^{-1}T$, viewed also as reduced
divisors on $\tC$. Then $\nu^*\omega_C=\omega_{\tC}(\tS)$
(\cite{Liu06}, Lemma 10.2.12(b)), and we have 
\begin{multline}\label{eq: chi_G omega_C T}
\chi_G(\omega_C(T)^{\otimes m}) = \chi_G(\nu^*\omega_C(T)^{\otimes m}) - [\sS\otimes  \omega_C(T)^{\otimes m}] \\
= \chi_G(\omega_{\tC}(\tS+\tT)^{\otimes m}) - [\sS\otimes  \omega_C(T)^{\otimes m}] 
\end{multline}
Since $S\cap T=\emptyset$, we have $\sS\otimes \omega_C(T)=
\sS\otimes\omega_C$. By Corollary~\ref{lem: conn decomp} and
Lemma~\ref{lem: S_P}, 
\begin{align*}
[\sS\otimes  \omega_C(T)^{\otimes m}] &=\sum_{P\in S}\frac{|G_P|}{|G|} \Ind_{G_P}^G [\sS_P]\otimes [\omega_C|_{P}]^{\otimes m} \\
& = \sum_{P\in S_1} \frac{|G_P|}{|G|}\Ind_{G_P}^G [1_{G_P}] +  \sum_{P\in S_2} \frac{|G_P|}{|G|}\Ind_{G_P}^G \chi_P^{\otimes m+1}. 
\end{align*}
%% where $\chi_P$ is defined as in \eqref{eq: chi_P}.
By the Frobenius reciprocity \eqref{eq: Frobenius2}, 
\begin{equation}\label{eq: S omega_C T G-inv}
\langle [1_{G}], [\sS\otimes  \omega_C(T)^{\otimes m}]\rangle =
\#\oS_1 + \#\oS_2\cdot \epsilon_m
\end{equation}
where $\oS_1=\pi(S_1)$.
%% , and $\epsilon_m$ is defined as in \eqref{eq: epsilon_m}.

Let $\nu_{D}\colon \tD\rightarrow D$ be the normalization,  and
let $\tilde \pi\colon \tC\rightarrow \tD=\tC/G$ be the quotient map.
Then $ \tD\setminus \tilde\pi(\tS\cup\tT)=D\setminus (\oS\cup
  \oT)$.
By Proposition~\ref{prop: log smooth}, we have
\begin{equation}\label{eq: omega_C T G-inv}
\langle [1_G], \chi_G(\nu^*\omega_C(T)^{\otimes m})\rangle =
\chi(\omega_{\tD}(\tilde\pi_*(\tS+\tT))^{\otimes m}) + \sum_{Q\in
  D\setminus (\oS\cup \oT)}\left\lfloor m\left(1-\frac{1}{|e_{Q}|}\right)\right\rfloor.    
\end{equation}
%% and $e_{\tQ}=|\oG_{\tP}|$ for $\tP\in \tilde \pi^{-1}(\tQ)$.
Note that, $\omega_{\tD}(\tilde\pi_*(\tS+\tT))^{\otimes m} =
\nu_D^*\omega_D(\oS_2+\oT)^{\otimes m}$, hence there is a short exact sequence
\[
0\rightarrow \omega_D(\oS_2+\oT)^{\otimes m} \rightarrow \nu_{D*}\omega_{\tD}(\tilde\pi_*(\tS+\tT))^{\otimes m} \rightarrow \sS_D\rightarrow 0
\]
where  $\sS_D = \oplus_{Q\in \oS_1} k_Q$ is the skyscraper sheaf supported on the singular locus $\oS_1$ of $D$. It follows that
\begin{equation}\label{eq: chi omega_tD}
\begin{split}
\chi(\omega_{\tD}(\tilde\pi_*(\tS+\tT))^{\otimes m}) & = \chi(\omega_D(\oS_2+\oT)^{\otimes m}) + \#\oS_1 \\
& = \chi(\omega_D(\oT)^{\otimes m}) + \#\oS_1 + m\#\oS_2.
\end{split}
\end{equation}
Combining the equalities \eqref{eq: chi_G omega_C T}, \eqref{eq: S omega_C T G-inv}, \eqref{eq: omega_C T G-inv}, and \eqref{eq: chi omega_tD}, we have
\begin{multline*}
\langle [1_G], \chi_G(\omega_C(T)^{\otimes m})\rangle  = \langle [1_G], \chi_G(\nu^*\omega_C(T)^{\otimes m})\rangle -\langle [1_{G}], [\sS\otimes  \omega_C(T)^{\otimes m}]\rangle \\
 =  \chi(\omega_D(\oT)^{\otimes m}) +(m - \epsilon_m)
\#\oS_2+\sum_{Q\in D\setminus (\oS\cup\oT)}\left\lfloor
  m\left(1-\frac{1}{|e_{Q}|}\right)\right\rfloor.     
\end{multline*}
%% where $\epsilon_m$ is as in \eqref{eq: epsilon_m}, and $e_Q:=|\oG_P|$ for $Q\in D\setminus{\oS\cup\oT}$ and $P\in\pi^{-1}(Q)$.
\end{proof}

\begin{cor}\label{cor: pluricanonical G-inv} 
Keep the hypothesis
    of Theorem~\ref{thm: pluricanonical G-inv}. Suppose further that $C$ is stable and $G\subseteq\Aut(C)$ is a subgroup of order prime to the 
  characteristic
  exponent of $k$.
  %% Let $T$ be a finite $G$-set contained in the smooth locus of
  %% $C$.
  Then for any positive integer $m$ such that $m+|T|\geq 2$, we have
\begin{multline}\label{eq: H^0 G-inv}
  \dim_k H^0(C, \omega_C(T)^{\otimes m})^G =
  \chi(\omega_D(\oT)^{\otimes m}) +(m - \epsilon_m) \#\oS_2 \\
+\sum_{Q\in D\setminus (\oS\cup\oT)}\left\lfloor m\left(1-\frac{1}{|e_{Q}|}\right)\right\rfloor  
\end{multline}
%%where $\pi\colon C\rightarrow D=C/G$ is the quotient map,
%%$\oS,\,\oT\subset D$ are the images of $S$ and $T$ respectively, for
%%$Q\in D\setminus(\oS\cup \oT)$, $e_Q:=|G_P|$ for $Q\in D\setminus
%%(\oS\cup\oT)$ and $P\in \pi^{-1}(Q)$, and
with $\epsilon_m =1$ if $m$ is odd and $0$ if $m$ is even.
\end{cor}

\begin{proof}
 Since $|G|\in k^*$, the isomorphism classes of $G$-modules $V$ are
 determined by their classes $[V]$ in $R_k(G)$. Also, we have
 $\#S_2=0$ if $\chara\, k=2$. The equality \eqref{eq: H^0 G-inv}
 follows from Theorem~\ref{thm: pluricanonical G-inv} and from
 the  fact that 
%% easy isomorphisms of $G$-modules:
 $H^1(C, \omega_C(T)^{\otimes m})=0$
 when $m+|T|\ge 2$ because $C$ is stable.
 \end{proof}

\subsection{The equivariant deformation space of a stable \texorpdfstring{$G$}{}-curve}\label{sec: Def(C,G)}\label{sec: deform}

Let the notation be as in Set-ups~\ref{setup: curve} and \ref{setup:
  nodal}. Suppose further that $C$ is stable and $|G|\in k^*$.
If $G$ acts freely on a dense open subset of $C$, then 
the equivariant deformation of  $G$-marked stable curves 
 $(C, G)$ is unobstructed, and the tangent space of the $G$-equivariant deformation space $\Def(C,G)$ is $\Ext^1(\Omega_C, \sO_C)^G$, where $\Omega_C$ denotes the sheaf of K\"ahler differentials on $C$ (\cite{Tuf93}). 
The more general case, where the action is possibly not free on a
dense open subset,  has been studied in \cite{Mau06, Cat12, Li21}.

Based on Theorem~\ref{thm: pluricanonical G-inv}, the following theorem
computes $ \dim \Ext^1(\Omega_C, \sO_C)^G$ for a stable $G$-curve $C$.

\begin{thm}\label{thm: Def(C,G)}
Let $S_3\subseteq S$ be the subset of  singular points that are
smoothable with respect to the $G$-action, as in Definition~\ref{defn:
  smoothable}. Let $B\subset D_\sm$ be the set of branch points of
$C_\sm\rightarrow \pi(C_\sm)$, $\oS_i\subset D$ be the image of $S_i$
for $1\leq i\leq 3$.  
Then
\begin{equation}\label{eq: Def(C,G)}
     \dim \Ext^1(\Omega_C, \sO_C)^G = 3(p_a(D)-1) + \#B+ 2\#\oS_2 
     -\#\oS + \#\oS_3.  
\end{equation}
\end{thm}

We first establish some preliminary results.

\begin{lem}\label{lem:OtOoS} Let $C$ be a nodal curve over $k$.
  Let $\sK_C$ be the sheaf of rational functions on $C$ and let
  $\sI\subset\sO_C$ be the sheaf of ideals defining the (reduced) singular
  locus $S$ of $C$. 
  \begin{enumerate}[\rm (1)] 
  \item The kernel of the canonical map
    $\Omega_C\to \Omega_C\otimes_{\sO_C} \sK_C$ is
    the torsion part $\Omega_{C, \tor}$ of $\Omega_{C}$ and
   $(\Omega_{C, \tor})_P$ is a simple $\sO_{C,P}$-module 
(hence isomorphic to $k(P)$.) 
  \item The image of $\Omega_C\to \Omega_C\otimes_{\sO_C} \sK_C$ is
    equal to $\sI \omega_C$.
   \item We have a canonical exact sequence
     \begin{equation}
       \label{eq:OtOoS}
      0 \to \Omega_{C, \tor}\to \Omega_C\to \omega_C \to \omega_C|_S
      \to 0.     \end{equation}
   \item 
The action of $G_{P}$ on $(\Omega_{C, \tor})_P\otimes \omega_C$ is trivial if and only if $(C, G)$ is smoothable at $P$.
  \end{enumerate} 
  \end{lem}

\begin{proof} (3) is an immediate consequence of (1) and (2). 
The statements of (1), (2) and (4) are local at singular points $P$
and are compatible with the extension to the
formal completion $\widehat{\sO}_{C,P}$. Therefore we can work
with the local $k$-algebra $k[[x,y]]/(xy)\cong \widehat{\sO}_{C,P}$, on which $G_P$ acts linearly. Then $(\Omega_{C, \tor})_P=k(xdy)=k(ydx)$. 
(1) and (2) are then straightforward.

(4) Note that $\omega_C|_P=k(dx/x)=k(dy/y)$ with $dx/x=-dy/y$.

(4.a) Let $\tau$ be a generator of $G_{\tP}$. We have
$\tau^*
(x)=\xi x$ and $\tau^*
(y)=\xi' y$ for some roots of unity $\xi, \xi'\in k^*$.
Then 
\[
\tau^*
(xdy)=\xi \xi' xdy,\quad \tau^*(dx/x)=dx/x.
\]
Therefore
$\tau$ acts trivially on $(\Omega_{C, \tor})_P\otimes \omega_C$ if and only
if $\xi'=\xi^{-1}$.

(4.b) Let $\sigma\in G_P\setminus G_{\tP}$ (this condition is
empty if $P\in S_1$). We know (Lemma~\ref{lem: S_P}) that $\sigma$ acts as
$-1$ on $\omega_C|_P$. On the other hand, there exist
$\lambda_\sigma, \lambda'_\sigma\in k^*$ such that 
\[ \sigma^*
(x)\in \lambda_\sigma y, \quad \sigma^*
(y)\in \lambda'_\sigma x. \]
This implies that
\[ \sigma^*
(xdy)=\lambda_\sigma \lambda'_\sigma ydx=-
(\lambda_\sigma \lambda'_\sigma) xdy\in (\Omega_{C, \tor})_P.\]  

Summarizing (4.a) and (4.b) we see that $G_P$ acts trivially on
$(\Omega_{C, \tor})_P\otimes \omega_C|_P$ if and only if  $\xi'=\xi^{-1}$ and
$\lambda'_\sigma=\lambda_\sigma^{-1}$, that is, the action of
$G_P$ at $P$ is smoothable. 
\end{proof}

\begin{lem} \label{lem:cext} Let $C$ be a projective nodal curve over $k$. 
  \begin{enumerate}[\rm (1)] 
\item Let $\sF$ be a
  coherent sheaf on $C$.  Then we have a canonical isomorphism
  \[
\Ext^1(\sF, \sO_C) \simeq H^0(C, \sF\otimes \omega_C)^{\vee}. 
  \] 
\item Let $G$ be a finite group of order $|G|\in k^*$ acting on $C$. 
  Suppose that $\sF$ is a skyscraper $G$-sheaf. Then 
  \[
  \dim \Ext^1(\sF, \sO_C)^G=\sum_{P\in C} \frac{|G_P|}{|G|}
  \dim (\sF_P\otimes \omega_{C})^{G_P}. 
  \] 
\item For any invertible sheaf $\sL$ on $C$ and any $i>1$,
  we have $\Ext^i(\sF, \sL)=0$.
 \end{enumerate} 
\end{lem}

\begin{proof} (1) As $\omega_C$ is invertible, we have
  $ \Ext^1(\sF, \sO_C) \simeq \Ext^1(\sF\otimes \omega_C, \omega_C)$ 
  (\cite[Prop. III.6.7]{Har71}). Then apply 
  %% As $C$ is Cohen-Macaulay (it is a reduced curve),
  Serre's duality theorem   (\cite[Theorem III.7.6]{Har71}).
%%    \[ \Ext^1(\sF\otimes \omega_C, \omega_C) \simeq H^0(C, \sF\otimes\omega_C)^{\vee}\]

  (2) Apply (1) and Corollary~\ref{lem: conn decomp}.

  (3) We have $\Ext^i(\sF, \sL)\simeq \Ext^i(\sF\otimes \sL^{\vee}\otimes \omega_C, \omega_C)$. So it is enough to show that $\Ext^i(\sF, \omega_C)=0$ 
  for all coherent sheaves $\sF$ and for all $i>1$. 
  
  For all $q \gg 0$ (depending only on $C$) and all $i>0$, we have
  \[\Ext^i(\sO_C(-q), \omega_C)\simeq \Ext^i(\sO_C, \omega_C(q))
  =H^i(C, \omega_C(q))=0.\]  
Consider the contravariant $\delta$-function $(T^i)_{i\geq 0}$
where $T^i=H^{1-i}(C, .)^{\vee}$ if $i=0,1$ and $T^i\equiv 0$ if $i\geq 2$. 
We can proceed exactly as for the proof of   Serre's duality (\cite[Theorem III.7.6]{Har71}) to show that the $\delta$-functor
$(\Ext^i(., \omega_C))_{i\geq 0}$ coincides with $(T^i)_{i\geq 0}$. Therefore
$\Ext^i(\sF, \omega_C)=0$ if $i>1$. 
\end{proof} 

\begin{proof}[Proof of Theorem~\ref{thm: Def(C,G)}] 
We compute the dimension of $\Ext^1(\Omega_C, \sO_C)^G$ by comparing
it with the dimension of $\Ext^1(\omega_C, \sO_C)^G$. 
By Lemma~\ref{lem:OtOoS}, we have an exact sequence. 
\[ 0 \to \Omega_{C, \tor}\to \Omega_C\to \sI \omega_C \to 0. \]
By \cite[Lemma~1.4]{DM69}, $\Ext^0(\Omega_C, \sO_C)=0$. Moreover
$ \Omega_{C, \tor}$ has finite support, hence $\Ext^0(\Omega_{C, \tor},
\sO_C)=0$.  Thus, taking the $\Ext^i(\cdot,\sO_C)$-groups 
we obtain $\Ext^0(\sI\omega_C, \sO_C)=0$ and the exact sequence 
\begin{equation}
  \label{eq:ext1}
  0 \to \Ext^1(\sI\omega_C, \sO_C) \to
  \Ext^1(\Omega_C, \sO_C) \to \Ext^1(\Omega_{C, \tor}, \sO_C) \to 0. 
\end{equation}
Similarly, taking $\Ext^i(., \sO_C)$ in the exact sequence 
$ 0 \to \sI\omega_C\to \omega_C \to \omega_C|_S\to 0$ 
we get the exact sequence
\begin{equation}
  \label{eq:ext2}
  0 \to \Ext^1(\omega_C|_S, \sO_C) \to \Ext^1(\omega_C, \sO_C)
  \to \Ext^1(\sI\omega_C, \sO_C) \to 0.  
\end{equation}
Combining 
Exact sequences \eqref{eq:ext1} and \eqref{eq:ext2}, and 
taking $G$-invariants (which is exact as $|G|\in k^*$), we get the
exact sequence 
\begin{multline*}
   0 \to \Ext^1(\omega_C|_S, \sO_C)^G \to \Ext^1(\omega_C, \sO_C)^G 
   \to \Ext^1(\Omega_C, \sO_C)^G \to  \\
   \to \Ext^1(\Omega_{C, \tor}, \sO_C)^G \to  0.
\end{multline*} 
It follows that
\begin{multline}\label{eq: Omega vs omega dim}
    \dim \Ext^1(\Omega_C, \sO_C)^G = \dim \Ext^1(\omega_C, \sO_C)^G -\dim \Ext^1(\omega_C|_S, \sO_C)^G \\
    + \dim \Ext^1(\Omega_{C, \tor}, \sO_C)^G.  
\end{multline}
By Lemma~\ref{lem:cext}(1) and Theorem~\ref{thm: pluricanonical G-inv}, we have
\[
\dim\Ext^1(\omega_C, \sO_C)^G=\dim\left(H^0(C, \omega_C^{\otimes
    2})^\vee\right)^G = 3(p_a(D)-1)+   \#B + 2\#\oS_2. 
\]
For each $P\in S$, by Lemma~\ref{lem:cext}(2), 
we have
\[ \dim \Ext^1(\omega_C|_S,
\sO_C)^G=\sum_{P\in S}\frac{|G_P|}{|G|} \dim (\omega_C^{\otimes 2}|_P)^{G_P}
=\# \oS.\] 
Similarly, the third term of \eqref{eq: Omega vs omega dim} is
\[ \dim \Ext^1(\Omega_{C, \tor}, \sO_C)^G
=\sum_{P\in S}\frac{|G_P|}{|G|}
  \dim ((\Omega_{C, \tor})_P \otimes \omega_C)^{G_P}  
  \]
  is 
exactly the number $\#\oS_3$ of $G$-orbits $G.P$  
of smoothable nodal points $P$ by Lemma~\ref{lem:OtOoS} (4) (Note that
  $(\Omega_{C, \tor})_P \otimes \omega_C$ has dimension $1$.)

Plugging these
observations into \eqref{eq: Omega vs omega dim}, we obtain the
desired formula \eqref{eq: Def(C,G)} for $ \dim \Ext^1(\Omega_C,
\sO_C)^G$. 
\end{proof}

\begin{rmk} 
    If $S=S_3$, that is, if $(C, G)$ is smoothable, then \eqref{eq: Def(C,G)} becomes a formula similar to the smooth case:
    \[
     \dim \Ext^1(\Omega_C, \sO_C)^G = 
     3(p_a(D)-1) + \#B+2\#\oS_2.
    \]
\end{rmk}

\subsection{Relations between \texorpdfstring{$H^1(C, \CC)$}{} and \texorpdfstring{$H^0(C, \omega_C)$}{}}
Let the notation be as in Set-up~\ref{setup: nodal}. In this subsection, we assume that the base field $k=\CC$. If $C$ is smooth, then the Hodge decomposition gives the desired relation between $H^1(C, \CC)$ and $H^0(C, \omega_C)$, as $G$-modules:
\[
H^1(C,\CC)\cong H^0(C, \omega_C) \oplus  H^1(C, \sO_C)\cong H^0(C, \omega_C) \oplus H^0(C, \omega_C)^\vee. 
\]
We explore relations between $H^1(C, \CC)$ and $H^0(C, \omega_C)$ when $C$ has at most nodal singularities.
There are two $G$-equivariant short exact sequences of sheaves of $\CC$-vector spaces:
\begin{equation}\label{eq: ses C_C}
    0\rightarrow \CC_C \rightarrow \nu_*\CC_{\tC} \rightarrow \oplus_{P\in S} \sS_P\rightarrow 0
\end{equation}
and
\begin{equation}\label{eq: ses omega_C}
      0\rightarrow  \nu_*\omega_{\tC}  \rightarrow \omega_C \rightarrow \oplus_{P\in S} \omega_C|_P\rightarrow 0  
\end{equation}
where $S$ is the set of singular points of $C$.

\begin{lem}
We have
\[
\left[H^1(C, \CC)\right] -\left[H^0(C, \omega_C)\right]  = \left[H^1(\tC, \CC)\right] - \left[H^0(\tC, \omega_{\tC})\right] = \left[ H^1(\tC, \sO_{\tC}) \right]
\]
\end{lem}

\begin{proof}
By \eqref{eq: ses C_C}, we have
\begin{multline}\label{eq: cohomology stable curve and normalization}    
    \left[H^0(C, \CC)\right] - \left[H^1(C, \CC)\right] + \left[H^2(C, \CC)\right] \\
    = \left[H^0(\tC, \CC)\right] - \left[H^1(\tC, \CC)\right] +
    \left[H^2(\tC, \CC)\right] - \left[\bigoplus_{P\in S}\sS_P\right]
\end{multline}
By \eqref{eq: ses omega_C}, we obtain
\begin{multline}\label{eq: hol cohomology stable curve and normalization}    
    \left[H^0(C, \omega_C)\right] - \left[H^1(C, \omega_C)\right] \\
    = \left[H^0(\tC, \omega_{\tC})\right] - \left[H^1(\tC, \omega_{\tC})\right] + \left[\bigoplus_{P\in S}\omega_C|_P\right]
\end{multline}
Observe that
\begin{equation}\label{eq: identify permutation representation}
\begin{split}
\left[H^2(C, \CC)\right] = \left[H^2(\tC, \CC)\right],\quad \left[H^0(C,\CC)\right]=\left[H^1(C, \omega_C) \right],\\
\left[H^0(\tC,\CC)\right] = \left[H^1(\tC, \omega_{\tC}) \right],\quad \left[\bigoplus_{P\in S}\sS_P\right] = \left[\bigoplus_{P\in S}\omega_C|_P\right]
\end{split}
\end{equation}
Combining \eqref{eq: cohomology stable curve and normalization},
\eqref{eq: hol cohomology stable curve and normalization}, and \eqref{eq: identify permutation representation}, we obtain the desired equalities
\[
\left[H^1(C, \CC)\right] -\left[H^0(C, \omega_C)\right]  = \left[H^1(\tC, \CC)\right] - \left[H^0(\tC, \omega_{\tC})\right] = \left[ H^1(\tC, \sO_{\tC}) \right]
\]
\end{proof}

\begin{cor}
    If the components of $C$ are all rational, then $H^1(C, \CC) = H^0(C, \omega_C)$ as $G$-modules.
\end{cor}

\begin{rmk}\label{thm: self dual}
For each $i\geq 0$, we have $H^i(C, \CC) \cong H^i(C, \CC)^\vee$ as $G$-modules, that is,  $H^i(C, \CC)$ is a self-dual $G$-representation. This is due to the fact that $G$ acts on $H^i(C, \ZZ)$, and $H^i(C, \CC) = H^i(C, \ZZ) \otimes_\ZZ \CC$, and real $G$-representations are self-dual (\cite[Theorem 31.21 and Definition 43.7]{CR62} and \cite[Section~3.5]{FH91}). 
\end{rmk}

\subsection{The induced action on the dual graph of a nodal \texorpdfstring{$G$}{}-curve}

Resume the notation from Set-up~\ref{setup: nodal}.
Let $C$ be a connected nodal curve over $k$, and $G\subseteq \Aut(C)$ a finite subgroup of automorphisms. One can define its dual graph $\Gamma=(V, E)$:
\begin{enumerate}
    \item Each irreducible component $C_i$ of $C$ corresponds to a vertex, denoted by $\langle C_i\rangle $.
    \item Each node $P$ of $C$ corresponds to an open edge, denoted by $\langle P\rangle $, connecting the vertices $C_i$ and $C_j$ such that $P \in C_i\cap C_j$. 
\end{enumerate}
We view $\Gamma$ as a cell complex with the $S$ as the set of the $0$-cells, and $\{C_i\mid 1\leq i\leq n\}$ as the set of $1$-cells. The cellular chain complex of $\Gamma$ reads
\[
C_\bullet(\Gamma): 0\rightarrow C_1(\Gamma) \rightarrow C_0(\Gamma)\rightarrow 0
\]
with $C_1(\Gamma)= \oplus_P\ZZ\langle P\rangle $ and $C_0(\Gamma) = \oplus_i \ZZ\langle C_i\rangle $. Note that there is an induced action of $\Aut(C)$ on $\Gamma$ and $C_\bullet(\Gamma)$. For a node $P\in C$, the stabilizer $G_{\langle P\rangle }$ of the edge $\langle P\rangle $ is the same as the stabilizer $G_P$ of the point $P$, and for $g\in G_P$, we have
\[
g(\langle P\rangle ) = 
\begin{cases}
    \langle P\rangle , & \text{ if $G_P$ preserves the local branches of $C\ni P$} \\
    -\langle P\rangle , & \text{ if $G_P$ interchanges the local branches of $C\ni P$}
\end{cases}
\]
This action does not make $\Gamma$ into a $G$-CW complex in the strict sense, because the points of the edges (or 1-cell) can have different stabilizers. Nevertheless, by the additivity of the Euler characteristic, we have an equality in the Grothendieck group $R_k(G)$:
\begin{equation}\label{eq: chain = homology}
    [C_0(\Gamma)\otimes_\ZZ k] - [C_1(\Gamma)\otimes_\ZZ k] = [H_0(\Gamma, k)] -  [H_1(\Gamma, k)].
\end{equation}

\begin{lem}
Let $\nu : \tC \rightarrow C$ be the normalization. Then there are
identifications of $G$-representations:
\begin{equation}\label{eq: chain = coherent homology}
    C_0(\Gamma)\otimes_\ZZ k\cong H^1(\tC, \omega_{\tC}),\quad C_1(\Gamma)\otimes_\ZZ k\cong \bigoplus_{P\in S} \omega_C|_P
\end{equation}
\end{lem}
\begin{proof}
Let $\{[C_{i_j}]\mid 1\leq j \leq l\}$ be a system of representatives of the $G$-orbits in the set of components $\{[C_i]\mid 1\leq i\leq n\}$. Then we have isomorphisms of $G$-representations
 \[
 H^1(\tC, \omega_{\tC}) = \bigoplus_{i} H^1(\tC_i, \omega_{\tC_i}) \cong \bigoplus_{1\leq j \le l} \Ind_{G_{[C_{i_j}]}}^G 1_{G_{[C_{i_j}]}} \cong  C_0(\Gamma)\otimes_\ZZ k.
 \]
 Let $\{P_j\mid 1\leq j\leq s\}$ be a system of representatives of the $G$-orbits in $S$. Then we have isomorphisms of $G$-representations
 \[
 \bigoplus_{P\in S} \omega_C|_P \cong\bigoplus_{j=1}^m\Ind_{G_{P_j}}^G \omega_C|_{P_j}  \cong C_1(\Gamma)\otimes_\ZZ k
 \]
\end{proof}
%Combining \eqref{eq: chain = homology} and \eqref{eq: chain = coherent homology}, we obtain an equality in $R_k(G)$:

%\begin{cor}$ [H_0(\Gamma, k)] -  [H_1(\Gamma, k)] = [H^1(\tC, \omega_{\tC})] -  [\bigoplus_{P\in S}\omega_C|_P].$\end{cor}

%\begin{ques}    Is there Chevalley--Weil formula for the $G$-action on the dual graphs of nodal curve or even more general graphs?\end{ques}

We may interpret the difference between $[H^0(C, \omega_C)]$ and $[H^0(\tC, \omega_{\tC})]$ using the action of $G$ on the dual graph $\Gamma$ of $C$.

\begin{cor}\label{cor: H^0(omega_C) vs H^0(omega_tC)}
    Let $C$ be a connected nodal curve and $G\subseteq\Aut(C)$ a
    finite subgroup. Let $\nu\colon \tC\rightarrow C$ be the
    normalization. Then we have the following equality in $R_k(G)$:
\[
[H^0(C, \omega_C)] = [H^0(\tC, \omega_{\tC})] + [1_G] - \chi_G(\Gamma, k )
\]
where $\chi_G(\Gamma, k ):= [H_0(\Gamma, k)] -  [H_1(\Gamma, k)]\in R_k(G)$.
\end{cor}

\begin{proof}
Taking the long exact sequence of cohomology groups 
associated to the short exact sequence \eqref{eq: omegaC-omegatC}
and noting that $[H^1(C,\omega_C)] = [1_G]$, we get
\[
  [H^0(C, \omega_C)] = [H^0(\tC, \omega_{\tC})] + [1_G] -
[H^1(\tC, \omega_{\tC})]+[\omega_C|_S]. 
\]
By \eqref{eq: chain = homology} and \eqref{eq: chain = coherent
  homology} we have $[H^1(\tC, \omega_{\tC})] -
[\omega_C|_S]=\chi_G(\Gamma,k)$. This proves the corollary.  
\end{proof}

\section{\texorpdfstring{$G$}{}-module structure on the singular cohomology of \texorpdfstring{$G$}{}-CW complexes}\label{sec: CW}
\begin{thm}\label{thm: CW for CW}
Let $X$ be an $n$-dimensional compact complex space, and
$G\subseteq \Aut(X)$ a finite subgroup of automorphisms. Let $Y=X/G$ be
the quotient space and $\pi\colon X\rightarrow Y$ the quotient
map. For each subgroup $H\subseteq G$, denote
\[
X_H:=\{ x \in X \mid \text{$G_ x$ conjugate to $H$} \} \, \text{ and }\, Y_H:=\pi(X_H).
\]
Let $H_1,\dots, H_s$ be a system of representatives of conjugacy classes of subgroups in $G$.
Then the following equality holds in 
$R_\QQ(G)$:
\begin{equation}\label{eq: CW for CS}
    \sum_{0\leq i\leq 2n}(-1)^i[H^i(X,\QQ)] = \sum_{1\leq j\leq s} \chi(Y_{H_j})\Ind_{H_j}^G [1_{H_j}].
\end{equation}
where for a complex space $Z$, $\chi(Z)$ denotes its topological Euler characteristic; if $Y_{H_j}=\emptyset$, we set $\chi(Y_{H_j})=0$.
\end{thm}
\begin{proof}
As in the proof of \cite[Theorem~2.3]{CL18}, there is a finite triangulation $\sT$ of $X$ such that for each simplex $\Delta\in \sT$ and for all $x\in \Delta^0$, we have 
\[
G_x = G_{\Delta^0} =G_{\{\Delta\}},
\]
where $\Delta^0$ denotes the relative interior of $\Delta$. Then, $X_{H_j}$ is a union of relatively open simplices of the triangulation $\sT$ for each $1\leq j \leq s$.

For each $0\leq i\leq 2n$, let $\sT_i\subseteq \sT$ be the set of $i$-dimensional simplices, and 
\[
C_i(\sT)=\bigoplus_{\Delta\in \sT_i} \QQ [\Delta]
\]
the $i$-th cellular chain group. Then there is an induced action of $G$ on  $C_i(\sT)$, and an equality in $R_\QQ(G)$:
\begin{equation}\label{eq: cohomology vs chain}
\sum_{0\leq i\leq 2n}(-1)^i[H^i(X,\QQ)]  = \sum_{0\leq i\leq 2n}(-1)^i [C_i(\sT)].    
\end{equation}
For each orbit of a simplex $\Delta\in\sT$, we obtain on the right-hand side of \eqref{eq: cohomology vs chain} a direct summand $\Ind_{H_j}^G [1_{H_j}]$, where $H_j\subseteq G$ is the subgroup conjugate to $G_{\{\Delta\}}$. Summing over all the orbits of $G$ in $\sT$ and sorting out according to the conjugacy classes of stabilizers, we obtain the desired equality \eqref{eq: CW for CS}.
\end{proof}

\begin{cor}\label{cor: chi Q}
Let $C=\bigcup_{1\leq i\leq n} C_i$ be a compact nodal curve over $\CC$, and
$G\subseteq\Aut(C)$ a finite subgroup of automorphisms. Let $\pi\colon
C\rightarrow D:=C/G$ be the quotient map. Let $S$ be the singular locus of $C$. Set $T:=\{P\in C\setminus S\mid
G_P\supsetneq I_i \text{ for } C_i\ni P\}$ and $A:=S\cup T$. Write $D=\bigcup_{1\leq j\leq m} D_j$ as the union of irreducible components. Then we have
\begin{equation}\label{eq: CW for constant sheaf}
    \sum_{i=0}^{2}(-1)^i[H^i(C,\QQ)] = \sum_{j=1}^m \chi(D_j\setminus\pi(A))\Ind_{I_i}^G \left[1_{I_i}\right] + \sum_{P\in A'} \Ind_{G_{P}}^G \left[1_{G_{P}}\right],
\end{equation}
where $A'$ denotes a set of representatives of the $G$-orbits in $A$.
\end{cor}

The following corollary is a direct generalization of \cite[Proposition~2]{Bro87}, which deals with the case when $C$ is a connected smooth curve.
\begin{cor}\label{cor: chi Q irreducible}
Let the notation be as in Corollary~\ref{cor: chi Q}. Suppose that $I_i$ is trivial for each component $C_i$. Then
\[
\sum_{i=0}^{2}(-1)^i[H^i(C,\QQ)] = \sum_{j=1}^m \chi(D_j\setminus\pi(A))\left[\QQ[G]\right] + \sum_{P\in A'} \Ind_{G_{P}}^G \left[1_{G_{P}}\right].
\]
In particular, if $C$ is smooth and connected, then
\[
[H^1(X,\QQ)] = (2g(D)-2+\#\pi(A))[\QQ[G]] + \sum_{P\in A'} \Ind_{G_{P}}^G \left[1_{G_{P}}\right]
\]
\end{cor}
%\begin{ques}    Can we replace $H^i(X, \QQ)$ by $H^i(X, \sF)$, where $\sF$ is a $G$-equivariant constructible sheaf (resp.~coherent sheaf) on $X$?\end{ques}

%\begin{rmk}There is a comparison between the Betti cohomology $H^*(X, \CC)$ and the \'etale cohomology $H^*_{\text{\'et}}(X,\ZZ_l)$. Using the Lefschetz principle and the character theory, one may translate the results about $[H^*(X, \CC)]_G$ to $[H^*_{\text{\'et}}(X)]_G$.\end{rmk}

\noindent{\bf Acknowledgements.} The first named author thanks
Olivier Brinon for instructive discussions on representation theory,
and Tianyuan Mathematical Center in Southeast China where part of this
work is done. 
The second named author would like to thank Professors Lie Fu and Qizheng Yin for discussions related to the Chevalley--Weil formula.  We would also like to thank a referee for pointing out an error in the proof of Theorem~\ref{thm: sm} in a previous version of the paper.

\end{document}